\documentclass[12pt]{article}
\title{\NTP{} topological structures}
\author{Pablo And\'ujar Guerrero}
\date{}

\usepackage{amsmath, amssymb, amsthm, mathtools}  % need for subequations, mathtools for \prescript
\usepackage{fullpage} 	% without this, will have wide math-article-paper margins
\usepackage{hyperref}
\usepackage{centernot}
\usepackage{xcolor}
\usepackage{bm} %for bold math
\usepackage[shortlabels]{enumitem} % for (a) labels in enumerate
\usepackage[normalem]{ulem} % to cross out with \sout

\newcommand{\Bb}{\mathcal{B}}
\newcommand{\Cc}{\mathcal{C}}

\newcommand{\Kk}{\mathcal{K}}
\newcommand{\Mm}{\mathcal{M}}
\newcommand{\Nn}{\mathcal{N}}

\newcommand{\Ff}{\mathcal{F}}
\newcommand{\Dd}{\mathcal{D}}
\newcommand{\Gg}{\mathcal{G}}

\newcommand{\Uu}{\mathcal{U}}

\newcommand{\Rr}{\mathcal{R}}

\newcommand{\Yy}{\mathcal{Y}}

\newcommand{\dF}{\ensuremath{F_\sigma^{\text{def}}}}
\newcommand{\dG}{\ensuremath{G_\delta^{\text{def}}}}
\newcommand{\BP}{\textbf{SHB}}

\newcommand{\Ralg}{\mathbb{R}_{\text{alg}}}
\newcommand{\R}{\mathbb{R}}

\newcommand{\BU}{\Bb} %Basis uniformity

\newcommand{\N}{\mathbb{N}}
\newcommand{\Z}{\mathbb{Z}}
\newcommand{\Q}{\mathbb{Q}}
\newcommand{\Qp}{\mathbb{Q}_p}

\newcommand{\Msh}{\Mm^{\text{Sh}}}
\newcommand{\Rshe}{\Rr^{\text{Sh}}}

\newcommand{\NIP}{\ensuremath{\mathrm{NIP}}}
\newcommand{\NTP}{\ensuremath{\mathrm{NTP}_2}}
\newcommand{\TP}{\ensuremath{\mathrm{TP}_2}}

\newcommand{\Mo}{\Mm^o}% open core 
\newcommand{\Ro}{\Rr^o}
\newcommand{\No}{\Nn^o}

\newcommand{\partn}[1]{\partial^{(#1)}}
\newcommand{\X}[1]{X^{(#1)}}

\newtheorem{theorem}{Theorem}[section] % numbered like the section
\newtheorem{lemma}[theorem]{Lemma}
\newtheorem{proposition}[theorem]{Proposition}
\newtheorem{corollary}[theorem]{Corollary}
\newtheorem{claim}{Claim}[theorem]
\newtheorem{fact}[theorem]{Fact}

\newtheorem{thmA}{Theorem}

\theoremstyle{definition}
\newtheorem{definition}[theorem]{Definition}
\newtheorem{example}[theorem]{Example}
\newtheorem{remark}[theorem]{Remark}
\newtheorem{question}[theorem]{Question}

\theoremstyle{remark}

\newenvironment{claimproof}[1][\proofname]
               {
                 \proof[#1]
                 
               }
               {
                 \endproof
               }

\begin{document}

\maketitle 

\begin{abstract}
A subset of a topological space is constructible if it is a finite Boolean combination of closed sets. We prove that every $\NTP$ expansion of $(\R,<,+)$ by constructible sets defines only constructible sets, and that definable functions are generically piecewise continuous. The result also holds for all $\NTP$ expansions of $(\Q_p,+,\cdot)$, and all $\NTP$ definably complete expansions of ordered groups. In the latter case, the structure is generically locally o-minimal, has definable choice, and carries a well-behaved notion of naive topological dimension. For $\NIP$ uniform topological structures, constructibility of definable sets is preserved in the Shelah expansion. We classify strong expansions of $(\R,<,+)$ by constructible sets, and obtain results on $\NTP$ d-minimal structures. 
\end{abstract}

%TODO: address and affiliation.

\noindent
{\small \emph{Mathematics Subject Classification 2020.} 03C45, 03C64 (Primary); 54H05 (Secondary). } 

\section{Introduction}

%This paper advances two related programs in topological model theory. On one hand, it delves into the exploration of tameness beyond o-minimality, centered mostly on expansions of the real line, known as \emph{Miller's program}~\cite{MS99, miller05-tame, DMS10}. On the other hand, it contributes to the recent breakthroughs in the investigation of the connections between neostability or tame combinatorics ---a model-theoretic setting put forward by Shelah--- and topological tameness~\cite{DG17, HW18, Wal22}, with a focus on general topological structures whose topology is generated by a definable uniformity.

This paper advances two closely related programs at the interface of topology and model theory. On one hand, it forwards the exploration of first-order structures over the real ordered group $(\R,<,+)$ known as \emph{Miller's program}~\cite{MS99, miller05-tame, DMS10}, which aims to identify meaningful dividing lines among these structures, governing the topological and geometric behavior of definable sets and functions. On the other hand, it contributes to the recent breakthroughs in the investigation of the connections between neostability ---Shelah's model-theoretic framework of tame combinatorics--- and topological tameness~\cite{DG17, HW18, Wal22}, with a focus on general topological structures whose topology arises from a definable uniformity. 
%Roughly speaking we show that, under mild combinatorial tameness assumptions, definable sets and functions in a broad class of topological structures are subject to strong topological constraints.

The main axiom of neostability that we will work with is the tree property of the second kind ($\NTP$), a setting generalizing both simple and $\NIP$ theories. A subset of a topological space is constructible if it is a finite Boolean combination of closed sets. Our first main result, Theorem~\ref{thm:main-intro}, shows that every $\NTP$ expansion of $(\R,<,+)$ or $(\Q_p, +, \cdot)$ by constructible sets defines only constructible sets. This improves previous work of Walsberg in the $\NIP$ setting~\cite{Wal22}, and answers the $\NTP$ case of what is known as \emph{the main conjecture on expansions of the real field}. This conjecture postulates that an expansion of $(\R,+,\cdot)$ by constructible sets defines a dense and codense subset of $\R$ iff it defines $\Z$. Our result yields that, if one such expansion is $\NTP$, then it cannot define a dense and codense subset of $\R$.

We call a structure $\Mm=(M,\ldots)$ a uniform (topological) structure if $M$ is a (Hausdorff) uniform space induced by a uniformity on $M\times M$ that admits a basis given by a definable family of sets. Our main examples of uniform structures are provided by expansions of topological groups. Given the induced topology on $M$ (and product topology on $M^n$), we say that $\Mm$ is constructible if every definable set is constructible. The open core $\Mo$ of $\Mm$ is defined as the reduct $\Mo$, in the sense of definability, generated by all definable closed sets (equivalently by all definable constructible sets). In particular, every continuous function $f:M^n \rightarrow M^m$ definable in $\Mm$ is definable in $\Mo$. In this framework, we prove the following. See Section~\ref{sec:prelim} for definitions. 

\begin{thmA}\label{thm:main-intro}
Let $\Mm$ be an $\NTP$ definably $\sigma$-compact uniform structure. Then the open core $\Mo$ is constructible. In particular, the following structures have constructible open core:
\begin{enumerate}
    \item  Any $\NTP$ definably complete expansion of an ordered group. For example, any $\NTP$ expansion of $(\R,<,+)$.
    \item For any prime $p$, any $\NTP$ expansion of the field of $p$-adic numbers $(\Q_p,+,\cdot)$.
\end{enumerate}
\end{thmA}

We prove Theorem~\ref{thm:main-intro} by introducing a new notion of being ``definably $F_\sigma$", which we denote $\dF$, and by showing that the $\NTP$ assumption implies that definable sets are $\dF$ iff they are constructible. Given this result, the rest of the proof generalizes ideas present in~\cite{MS99, DMS10}. Specifically, our notion of definable $\sigma$-compactness ensures that $\dF$ sets are closed under projections, and from this and the aforementioned equivalence it is easy to see that definable constructible sets form a structure, namely the open core.

Our second main result further explores the assumptions in Theorem~\ref{thm:main-intro}, improving, in this context, the conclusion of a theorem of Jayne and Rogers~\cite{Jayne-Rogers} on first level Borel functions. Specifically, it establishes that definable functions in the open core are both generically and (finitely) piecewise continuous. 

\begin{thmA}\label{thm:f-intro}
Let $\Mm$ be an $\NTP$ definably $\sigma$-compact uniform structure. Let $f:X\rightarrow Y$ be a function definable in the open core. Let $X_0=X$ and, for every $n<\omega$, let $X_{n+1}$ be the closure in $X_n$ of the set of points of discontinuity of $f|_{X_n}$. Then:
\begin{enumerate}
     \item $X_j$ has empty interior in $X_i$ for every $i<j$ (generic continuity).
     \item There exists an $m$ such that $X_m=\emptyset$ (piecewise continuity).  
\end{enumerate}
\end{thmA}

Both Theorems~\ref{thm:main-intro} and~\ref{thm:f-intro} can be derived from a condition apparently weaker than $\NTP$, which we call \emph{strongly hereditarily Baire (\BP)}. In the case of expansions $\Rr$ of $(\R,+,\cdot)$, we are able to recast the previous two theorems into a dichotomy of descriptive set theoretic flavour (Theorem~\ref{thm:Baire-class-one}). That is, either the open core $\Ro$ is constructible and every function definable in it is generically piecewise continuous, or otherwise $\Ro$ defines a first level Borel (in particular Baire class one) function that is not a finite union of continuous functions.

We investigate the smaller class of $\NIP$, rather than $\NTP$, uniform structures. Many $\NIP$ classes of topological structures (e.g. o-minimal, P-minimal) are known to be constructible. Given a structure $\Mm$, its Shelah expansion is the expansion of $\Mm$ generated by all externally definable sets. Shelah expansions are central objects of study in $\NIP$ model theory. %Many $\NIP$ uniform structures are known to be constructible, including the open core of the structures described in Theorem~\ref{thm:main-intro}. 

\begin{thmA}\label{thm:nip-intro}
Let $\Mm$ be an $\NIP$ uniform structure. If $\Mm$ is constructible, then its Shelah expansion $\Msh$ is constructible too. 
\end{thmA}

Motivated by Theorem~\ref{thm:main-intro}, we investigate constructible expansions of ordered groups, and specifically of $(\R,<,+)$. A linearly ordered structure $\Mm$ is definably complete if every bounded nonempty unary definable set has a supremum in $M$. In this context, we establish the following.

\begin{thmA}\label{thm:cons-tame-top}
Let $\Mm=(M,<,+,\ldots)$ be a constructible definably complete expansion of an ordered group. Then:
\begin{enumerate}
    \item $\Mm$ is generically locally o-minimal (Proposition~\ref{prop:generic-local-o-min}).
    \item  $\Mm$ has definable choice (Proposition~\ref{prop:choice}).
    %\item If $M=\R$, then the topological dimension\footnote{That is, small inductive dimension, large inductive dimension, or Lebesgue covering dimension, all of which are identical for subsets of $\R^n$.} of every definable set coincides with its naive dimension (Proposition~\ref{prop:dim}).
\end{enumerate}
\end{thmA}

We also show, in Proposition~\ref{prop:dim}, that topological dimension is well-behaved. 

We then investigate the smaller class of strong, rather than $\NTP$, structures. In this setting we generalize the description of strongly dependent expansions of the real ordered group due to Walsberg~\cite{Wal22}. 

\begin{thmA}\label{thm:strong-intro}
Let $\Rr$ be a strong definably complete expansion of an ordered group. Then the open core $\Ro$ is locally o-minimal. 

If $\Rr$ is a strong expansion of $(\R,<,+)$, then $\Ro$ is either o-minimal or interdefinable with $(\R,<,+,\Bb, \alpha\Z)$, for some $\alpha>0$, where $\Bb$ is a collection of bounded Euclidean sets satisfying that $(\R,<,+,\Bb)$ is o-minimal. In particular $\Ro$ is constructible and has dp-rank $2$. 
\end{thmA}

Further results for strong uniform structures include showing, in Theorem~\ref{thm:strong-noise}, that these are constructible iff they are strongly noiseless (they do not define sets $Y\subseteq X$ where $Y$ is dense and codense in $X$), and that, among them, having constructible open core is a property maintained under elementary equivalence (Proposition~\ref{prop:cons-OC}). 

An open question in the classification of $\NTP$ (and even $\NIP$) expansions of $(\R,<,+)$ is whether these satisfy that every nowhere dense unary definable set is the union of finitely many discrete sets (d-minimality). We end the paper by delving into this subject, proving Theorem~\ref{thm:d-minimal-intro}.

\begin{thmA}\label{thm:d-minimal-intro}
Let $\Rr$ be an expansion of an ordered group. 
\begin{enumerate}
\item If $\Rr$ is $\NIP$ and d-minimal then its Shelah expansion $\Rshe$ is d-minimal (Proposition~\ref{prop:shelah-d-min}). 
\item If $\Rr$ is $\NTP$ then it cannot define a set $X\subseteq R$ such that $(X,<)$ is order-isomorphic to an ordinal $\alpha\geq \omega^\omega$ (Proposition~\ref{prop:omega}). 
\end{enumerate}
\end{thmA}

The structure of the paper is as follows. In Section~\ref{sec:prelim} we include the necessary preliminaries. In Section~\ref{sec:cons-df} we introduce our notion of $\dF$, and prove crucial results about definable constructible sets, showing that non-constructibility is equivalent to being dense and codense in a type-definable set. In Section~\ref{sec:BP} we present our notion of being strongly hereditarily Baire (\BP) and prove versions of Theorems~\ref{thm:main-intro} and~\ref{thm:f-intro} assuming this property in place of $\NTP$. In Section~\ref{sec:ntp} we investigate $\NTP$ uniform structures, proving that these are $\BP$. In Section~\ref{sec:nip} we prove Theorem~\ref{thm:nip-intro} about $\NIP$ uniform structures. In Section~\ref{sec:LO-structures} we focus on definably complete expansions of ordered groups. We complete the proof of Theorem~\ref{thm:main-intro}, recast Theorems~\ref{thm:main-intro} and~\ref{thm:f-intro} into a descriptive set theoretic dichotomy (Theorem~\ref{thm:Baire-class-one}), and prove Theorem~\ref{thm:cons-tame-top}. In Section~\ref{sec:strong} we prove Theorem~\ref{thm:strong-intro}, as well as our other results on strong uniform structures. Finally, in Section~\ref{sec:d-minimal} we prove Theorem~\ref{thm:d-minimal-intro} about d-minimality. 

\subsection*{Acknowledgments}

The author thanks Erik Walsberg, Will Johnson, Philipp Hieronymi, John Goodrick, and Antongiulio Fornasiero, for helpful conversations on the topics of this paper. The author was supported by the UK Engineering and Physical Sciences Research Council (EPSRC) Grant EP/V003291/1. Part of this research was carried out while the author was hosted at the Hausdorff Research Institute for Mathematics, funded by the Deutsche Forschungsgemeinschaft (DFG, German Research Foundation) under Germany's Excellence Strategy – EXC-2047/1 – 390685813. 

%He also thanks the organizers of the 2025 INdAM mini-workshop ``Recent developments in the model theory of fields" at the University of Naples Federico II, the 2025 conference ``O-Minimal Geometry - Interactions, Applications and Wider Developments" at the University of Warwick, and the 2025 trimester ``Definability, decidability, and computability" at the Hausdorff Research Institute for Mathematics (HIM), for invitations to share and discuss the material being presented here. 

%Please share your preprints and publications with our outreach manager at stefan.hartmann@hcm.uni-bonn.de.

\section{Preliminaries} \label{sec:prelim}

\subsection{Model-theoretic conventions and terminology}

Throughout we denote (first-order) structures by calligraphic letters $\Mm$, $\Rr$, $\ldots$, and their respective universes by uppercase roman letters $M$, $R$, $\ldots$. We denote tuples of variables or parameters by $x,y, z, a, b, c,\ldots$, and non-negative integers by $n, m, i, k, \ldots$. Given a structure $\Mm$, ``definable" means ``definable in $\Mm$ possibly with parameters" unless otherwise specified. A type-definable set is an intersection of definable sets. 

Given a (partitioned) formula $\varphi(x,y)$ in the language of $\Mm$, and parameters $b\in M^{|y|}$, we let $\varphi(M,b)=\{ a\in M^{|x|} : \Mm\models \varphi(a,b)\}$. A family $\{X_b : b\in B\}$ of subsets of $M^n$ is definable if $B\subseteq M^m$ is a definable set and there exists a formula $\varphi(x,y)$, with $|x|=n$ and $|y|=m$, such that $X_b=\varphi(M,b)$ for every $b\in B$. 
A collection of formulas is $k$-inconsistent if every sub-collection of size $k$ is inconsistent. We adopt the standard convention of identifying formulas with the sets that they define in a given background structure and vice versa.

\subsection{Neostability}

We present the main notions from neostability that we will use. For other relevant definitions such as $\NIP$, strong dependence, and dp-rank, we direct the reader to $\cite{simon:guide-NIP}$.

A structure $\Mm$ has the \emph{tree property of the second kind}, denoted \emph{$\TP$}, if there exists a formula $\varphi(x, y)$, some $k$ and, for every $n$ and $m$, a collection of parameters \mbox{$\{b_{i,j} : 1\leq i\leq n, \, 1\leq j \leq m\}$} in $M^{|y|}$, such that
\begin{enumerate}[(1)]
    \item $\{\varphi(x , b_{i,j}) : j \leq m\}$ is $k$-inconsistent for every $i\leq n$,
    \item $\{\varphi(x , b_{i,\eta(i)}) : i\leq n\}$ is consistent for every function $\eta : \{1,\ldots, n\} \rightarrow \{1,\ldots, m\}$.
\end{enumerate}
Equivalently, in the above definition one may always assume that $k=2$.
%Corollary 3.6 in ``Dense codense predicates and NTP" explains that 2-inconsistency in the definition of ntp2 is equivalent to k-inconsistency. 
A structure is $\NTP$ if it does not have $\TP$. A theory $T$ is $\NTP$ if any (every) model is $\NTP$. 
%For this, one may always check conditions (1) and (2) in an $\omega_1$-saturated structure, where it suffices to take $n=m=\omega$. 

%Given a set $X\subseteq M^{|x|}$, we say that $X$ has $\TP$ if the conditions above are satisfied and, moreover, in condition (2), the family $\{\varphi(x, a_{i,\eta(i)}) : i\leq n\}$ is consistent with $X$.   

For any cardinal $\kappa$, an \emph{inp-pattern of depth $\kappa$} in $\Mm$ is a family of pairs $
(\varphi_\alpha(x,y_\alpha), k_\alpha)
$,
for $\alpha < \kappa$, where $\varphi_\alpha(x,y_\alpha)$ is a formula and $k_\alpha$ a positive integer satisfying that, for every finite subset $\{ \alpha_1,\ldots, \alpha_n\} < \kappa$ and every $m$, there are parameters \mbox{$\{b_{i,j} : 1\leq i\leq n, \, 1\leq j \leq m\}$}, with $b_{i,j}\in M^{|y_{\alpha_i}|}$, such that
\begin{enumerate}[(1)]
    \item $\{\varphi_{\alpha_i}(x , b_{i,j}) : j \leq m\}$ is $k_{\alpha_i}$-inconsistent for every $i\leq n$,
    \item $\{\varphi_{\alpha_i}(x, b_{i,\eta(i)}) : i\leq n\}$ is consistent for every function $\eta : \{1\ldots, n\} \rightarrow \{1,\ldots, m\}$.
\end{enumerate}
We say that a set $X\subseteq M^{|x|}$ admits an inp-pattern of depth $\kappa$ if the conditions above are satisfied and, in condition (2), each family $\{\varphi_{\alpha_i}(x, b_{i,\eta(i)}) : i\leq n\}$ is consistent with $X$. The burden of $\Mm$ is the supremum of the depths of inp-patterns admitted by $M$. %Burden definition taken from Chernikov, who quotes Adler. We still define Burden, even though we only mention it a few times, because it is straightforward. 
We say that $\Mm$ is \emph{strong} if every inp-pattern in $\Mm$ has finite depth. A theory $T$ is strong if any (every) model is strong. Every strong theory is $\NTP$.
%A theory $T$ has \TP{} if every (any) model has an inp-pattern of depth $\kappa$, for any given ordinal $\kappa$.

Every strongly dependent theory is strong and $\NIP$, and every $\NIP$ theory is $\NTP$. %An $\NIP$ theory is strong iff it is strongly dependent. 

\subsection{Topological conventions and uniform structures}
%\subsection{Topological conventions and terminology}

Let $X$ be a subset of a topological space. We denote by $cl(X)$, $int(X)$, and $\partial X = cl(X)\setminus X$ the closure, interior and frontier of $X$ respectively. For any $n$, we define $\partial^{(n)}X$ recursively by letting $\partial^{(0)}X=X$ and $\partial^{(n+1)}X= \partial \partn{n}X$. Recall that a subset $Y$ of a topological space is nowhere dense if $cl(Y)$ has empty interior. Given two subsets $X$, $Y$ of a topological space, we say that $Y$ is dense in $X$ if $cl(Y\cap X) \supseteq X$, and say $Y$ is codense in $X$ if $cl(X\setminus Y) \supseteq X$. If $X\subseteq Y$, we refer to the interior of $X$ in $Y$ always with respect to the subspace topology on $Y$. 

%\subsection{Uniform structures} 

A (first-order) \emph{topological structure} is a structure $\Mm$ together with a Hausdorff\footnote{The Hausdorff assumption is made for convenience; many of the results in this paper can be easily adapted to the non-Hausdorff case.} topology on its universe $M$ that has a basis given by a definable family of sets. A \emph{uniform structure}\footnote{Unfortunately, the term ``uniform structure" is also common terminology to refer to a uniformity. The reader should take caution to avoid this confusion.} is a structure $\Mm$ together with a uniformity on $M$ that admits a basis $\BU$ given by a definable family of sets. More precisely, $\BU$ is a definable collection of subsets of $M^2$ which satisfies:
\begin{enumerate}[label=(\alph*)]
    \item $\bigcap \BU = \{ (x,x) : x\in M\}$,
    \item if $U\in \BU$ and $(x,y)\in U$, then $(y,x)\in U$ (symmetry), 
    \item for any $U, V\in \BU$, there exists $W\in \BU$ such that $W\subseteq U\cap V$ (downward directedness), 
    \item for any $U\in \BU$, there exists $V\in \BU$ such that $V\circ V\subseteq U$, where 
    $$V\circ V=\{ (x,y) \in M^2 : \exists z\in M,\, \{(x,z), (z,y)\}\subseteq V\}.$$ 
\end{enumerate}

The sets in $\BU$ are called basic entourages. Given any $x\in M$ and $U\in \BU$, we define the $U$-ball of center $x$ as $U[x]=\{ y\in M : (x,y)\in U\}$. Similarly, for any set $X\subseteq M$, we define $U[X]=\bigcup_{x\in X} U[x]$. The set $M$ is given the topology where, for every $x\in M$, the family of balls $\{ U[x] : U\in \BU\}$ is a basis of neighborhoods of $x$. Hence every uniform structure is a topological structure. For any $n$, we put on $M^n$ the product uniformity, which has basis $\Bb(n)=\{ U_1\times \cdots \times U_n : U_i\in \BU, \, i\leq n\}$, and induces the product topology on $M^n$. Given a definable set $Y\subseteq M^n$, the subspace uniformity also has a definable basis $\{ V \cap (Y\times Y) : V\in \Bb(n)\}$. Condition (a) ensures that the topology is Hausdorff. Condition (b) is sometimes substituted with a weaker property (see e.g.~\cite{dolich-good22}), however there is no loss of generality in assuming it. By passing from $\BU$ to $\{ int(U) : U\in \BU\}$ if necessary, we may assume that basic entourages are open, which implies in particular that \textbf{balls $U[x]$ are always open}.

Uniform structures have been explored, largely in the context of dp-minimality, in~\cite{sim-wals19, dolich-good22}. 
Two major classes of examples are given by $\Gamma$-valued metric spaces, for $\Gamma$ any ordered abelian group, and by topological groups (see \cite{sim-wals19}). We are particularly interested in ordered groups and valued fields.  
\begin{example}
 \begin{enumerate}
     \item Let $\Gg=(G,0,+,<,\ldots)$ be an expansion of an ordered group. Then $\Gg$ is a uniform structure, with basic entourages of the form $U_b=\{ (x,y)\in G^2 : |x-y|<b\}$, for $b\in G^{>0}$. 
     %Totally ordered group is a topological group. Every topological group is a uniform space in a natural way, by translations of neighborhoods of zero. 
     \item An expansion $\Kk$ of a valued field with, say, the divisibility predicate $v(x)\leq v(y)$, is a uniform structure, with basic entourages of the form $U_b=\{ (x,y)\in K^2 : v(b) \leq v(x-y)\}$, for $b\in K^\times$.   
 \end{enumerate}   
\end{example}

Given $\Mm$ a topological structure, the \emph{open core} of $\Mm$, denoted $\Mo$, is the reduct of $\Mm$ generated by all closed (open) definable sets, in the language with a symbol for every nonempty definable open set. This notion was first introduced in~\cite{MS99} and further explored in~\cite{DMS10}, although the idea of analyzing topological structures whose basic predicates are open or closed can be traced back to Robinson~\cite{rob73}. Given a uniform structure $\Mm$, \textbf{we assume throughout that $\Mo$, endowed with the same topology, is also a uniform structure}. This assumption is made because it is not clear to the author whether a definable basis for the uniformity can always be found in $\Mo$. Nevertheless, whenever the uniformity is given by a definable order $<$ or valuation $v$, the assumption is redundant, since the relations given respectively by $x < y$ and $v(x)\leq v(y)$ are both open in $M^2$. 

%\subsection{Examples of \NTP{} uniform structures}

Our assumption of definable $\sigma$-compactness (Definition~\ref{def:compact}) in Theorems~\ref{thm:main-intro} and~\ref{thm:f-intro}, made in place of properties invoking classical topological compactness, is a meaningful one, since it includes, for example, definably complete expansions of ordered groups (Lemma~\ref{lem:intervals-comp}). Our assumption of $\NTP$ in place of $\NIP$ is also relevant. Below we include some examples to illustrate this point.  

\begin{example}
\begin{enumerate}
    \item Let $\Uu$ be a non-principal ultrafilter on the set of prime numbers $P$. By \cite[Example 1.2]{ntp2-groups-fields}, the ultraproduct $\prod_{p\in P} \Q_p/\Uu$ is $\NTP$, but neither simple nor $\NIP$. Furthermore, it is definably $\sigma$-compact. This is a direct consequence of the fact that definable $\sigma$-compactness is expressed by a fixed set of sentences in the theory of any topological structure, which depends only on the formula defining the topology and the formula defining the witnessing upward directed covering. 

    \item An expansion of an o-minimal group by a generic predicate is inp-minimal, hence strong, but the theory has the independence property~\cite[Section 3.4]{DG17}. In this case the open core is o-minimal by~\cite{DMS10}.
    
    \item Dolich and Goodrick~\cite[Section 3]{DG-burden2} build a definably complete expansion of a divisible ordered Abelian group that has burden $2$ and the independence property. 
\end{enumerate}    
\end{example}

\section{Constructible and $\dF$ sets} \label{sec:cons-df}

Recall that a subset of a topological space is $F_\sigma$ if it is a countable union of closed sets. We introduce a definable analogue of this property for definable sets in topological structures. 

%In general topology, a subset of a topological space is $F_\sigma$ if it is the union of a countable family of closed sets and $G_\delta$ if it is the intersection of a countable family of open sets. We work with definable analogues of these properties. 

A family of sets $\Cc$ is \emph{upward} (respectively \emph{downward}) \emph{directed} if, for any $C_1, C_2 \in \Cc$, there exists some $C_3\in \Cc$ such that $C_1 \cup C_2 \subseteq C_3$ (respectively $C_3 \subseteq C_1 \cap C_2$). Equivalently, if, for any finite subfamily $\Ff\subseteq \Cc$, there is $C_3\in \Cc$ such that $\bigcup \Ff \subseteq C_3$ (respectively $C_3\subseteq \bigcap\Ff$).

\begin{definition}
Let $X$ be a definable set in a topological structure $\Mm$. We say that $X$ is $\dF$ (definably $F_\sigma$) if $X$ is the union of an upward directed definable family of closed sets. 
%We say that $X$ is $\dG$ if its complement is $\dF$, equivalently if $X$ is the intersection of a downward directed definable family of open sets. 
\end{definition}

%Note: rational numbers are a $F_\sigma$ set but not $\dF$ set in $(\mathbb{R},+,\cdot,<)$.

Similar but more restrictive notions than $\dF$ are called $D_\Sigma$ in~\cite{DMS10} or simply $\mathcal{F}_\sigma$ in~\cite{def-baire}. 
%Proposition~\ref{prop:dim} uses the notion of $D_\sigma$, and the proof of Theorem~\ref{thm:Baire-class-one} relies implicitly on the notion of $\mathcal{F}_\sigma$. 

A subset of a topological space is \emph{constructible} if it is a finite Boolean combination of closed (open) sets. A topological structure $\Mm$ is \emph{constructible} if every definable set in it is constructible. Robinson~\cite{rob73} and Pillay~\cite{pillay87} studied topological structures with various constructibility assumptions on them. 

The following characterization of constructible sets follows from the main result in~\cite{MR01} and will be assumed and used often throughout.  

\begin{fact}[\cite{MR01}]\label{fct:constructible0}
Let $X$ be a subset of a topological space. The following are equivalent. 
\begin{enumerate}[(1)]
    \item $X$ is constructible. 
    \item $\partial^{(n)}X = \overbrace{\partial \cdots \partial}^{n} X=\emptyset$ for some $n<\omega$. 
\end{enumerate}
\end{fact}

Fact~\ref{fct:constructible0} has the following consequence. 

\begin{fact}[\cite{MR01}]\label{fct:constructible1}
Let $X$ be a definable set in a topological structure. Then $X$ is constructible if and only if it is a finite Boolean combination of closed (open) \textbf{definable} sets.
\end{fact}

Fact~\ref{fct:constructible1} states that a definable set in a topological structure $\Mm$ is constructible iff it is definable without quantifiers in the open core $\Mo$. In particular, the open core can be defined as the reduct generated by definable constructible sets, and $\Mo$ is constructible iff it has quantifier elimination. 

\begin{remark}\label{rem:extensions}
\begin{enumerate}
\item If a definable set in a topological structure $\Mm$ is $\dF$, then it is clearly still $\dF$ after passing to its definition in any elementary extension of $\Mm$. By Fact~\ref{fct:constructible1}, the same is true of being constructible. We shall use these facts without further explanation. 

\item Since fibers of constructible sets (in Hausdorff spaces) are constructible, a topological structure is constructible iff all $0$-definable sets are constructible. By Fact~\ref{fct:constructible1}, it follows that constructibility is closed under elementary equivalence. 
\end{enumerate}
\end{remark}

The following lemma is well known in specialized settings for notions such as $D_\Sigma$ (see Remark~\ref{rem:simple-BP}) and $\Ff_\sigma$ in place of $\dF$. It constitutes a crucial reason why we work mostly in the context of uniform, rather than general, topological structures. 

\begin{lemma}\label{lem:cons-dF}
Let $X$ be a constructible definable set in a uniform structure $\Mm$. Then $X$ is $\dF$. 
%Then $X$ is both $\dF$ and $\dG$. 
\end{lemma}
\begin{proof}
Let $X\subseteq M^n$ and $\BU$ be a definable basis for the product uniformity on $M^n$. If $X=A\cap B$, where $A$ and $B$ are definable sets with $A$ open and $B$ closed (i.e. $X$ is locally closed), then the family of sets of the form $B \setminus U[M^n\setminus A]$, for $U\in \BU$, witnesses that $X$ is $\dF$. It is easy to prove that a finite union of $\dF$ sets is $\dF$ and so, by Fact~\ref{fct:constructible1}, the lemma follows.  
\end{proof}

%Suppose that $\Mm$ expands an ordered abelian group. Then the topology on $M^n$ is generated by a definable ($M$-valued) metric $d$. By changing, in the proof of Lemma~\ref{lem:cons-dF}, the $U$-ball $U[M^n\setminus A]$ to $\{ x\in M^n : d(x,M^n\setminus A)<r\}$, for $r\in M^{<0}$, one may show that the family witnessing $\dF$ can always be chosen to be nested and indexed by $r \in M^{>0}$. This corresponds to the notion of $\Ff_\sigma$ set used in~\cite{def-baire}. Similarly, in a valued field the witnessing family can be chosen to be indexed by the value group. 
%See Lemma 3.2 in ``Definably complete Baire structures and Pfaffian closure" or ``Definably complete Baire structures".

\begin{remark}\label{rem:dG}
Recall that the complement of an $F_\sigma$ set is called $G_\delta$. Accordingly, we might call the complement of a $\dF$ set a $\dG$ set. Since constructible sets are closed under complements, by Lemma~\ref{lem:cons-dF} every constructible definable set in a uniform structure is also $\dG$, in agreement with the general theory of uniform spaces. 
\end{remark} 

The next two lemmas establish the crucial relationship between non-constructible sets and sets that are dense and codense in some set. We regard this straightforward connection as a central insight of this paper. 

\begin{lemma}\label{lem:cons-noise0}
Let $X$ be a subset of a topological space. If there exists another subset $Z\neq \emptyset$ such that $X$ is dense and codense in $Z$, then $X$ is not constructible. 
\end{lemma}
\begin{proof}
Let $\partial_Z$ refer to the frontier with respect to the subspace topology on $Z$. If $X$ is dense and codense in $Z$ then, in particular, we have that $\partial_Z^{(n)}(X\cap Z) \neq \emptyset$ for every $n$. Note that $\partial_Z^{(n)}(X\cap Z) \subseteq \partial^{(n)}X$ for every $n$. Hence the lemma follows from Fact~\ref{fct:constructible0}.
\end{proof}

We show that, as long as our set $X$ is definable in an $\omega_1$-saturated topological structure, the converse of Lemma~\ref{lem:cons-noise0} also holds. 

\begin{lemma}\label{lem:cons-noise1}
Let $X$ be a definable set in an $\omega_1$-saturated topological structure. If $X$ is not constructible, then $X$ is dense and codense in the nonempty type-definable set $\bigcap_{n<\omega} cl(\partial^{(n)}X)$.
\end{lemma}
\begin{proof}
Observe that the sequence of definable sets $cl(\partial^{(n)}X)$ is decreasing. By Fact~\ref{fct:constructible0}, each set $cl(\partial^{(n)}X)$ is nonempty. So $C=\bigcap_{n<\omega} cl(\partial^{(n)}X)$ is a nonempty type-definable set. We note that $X$ is dense and codense in $C$. 

First observe that $X\cap C$ and $X\setminus C$ are both type-definable sets. In particular we have $X\cap C=\bigcap_{n<\omega} \partial^{(2n)} X$ and $C\setminus X = \bigcap_{n<\omega} \partial^{(2n+1)} X$, where $(\partial^{(2n)} X)_n$ and $(\partial^{(2n+1)} X)_n$ are both clearly decreasing sequences of nonempty sets. So every $x\in C\setminus X$ is in $cl(\partial^{(2n)} X)$ for every $n$, and so, by saturation, $x\in cl(X\cap C)$. Similarly, one shows that $X\cap C \subseteq cl(C\setminus X)$.  
\end{proof}

Joining Lemmas~\ref{lem:cons-noise0} and~\ref{lem:cons-noise1} (and applying Remark~\ref{rem:extensions}) we get that a definable set is not constructible iff, after passing to its definition in an $\omega_1$-saturated elementary extension, the set is dense and codense in a type-definable set. 

A topological structure $\Mm$ is \emph{noiseless} if there does not exist a definable set that is dense and codense in a (definable) open set. Furthermore, $\Mm$ is \emph{strongly noiseless} if there does not exist a definable set that is dense and codense in another definable set. These properties are central to Miller's program~\cite{miller05-tame}. Lemmas~\ref{lem:cons-noise0} and~\ref{lem:cons-noise1} highlight the close connection between them and constructibility. In particular, the following implications are obvious:
\[
\text{Constructible} \Rightarrow \text{Strongly Noiseless} \Rightarrow \text{Noiseless}.
\]

We end the section with a short proposition that might be of independent interest.

\begin{proposition}\label{dense-filter}
Let $X$ be a nonempty topological space. The constructible dense subsets of $X$ form a filter in the Boolean algebra of constructible sets. 

If $X$ is a nonempty set definable in a constructible structure, then the definable dense subsets of $X$ form a (consistent, partial) type. 
\end{proposition}
\begin{proof}
The second paragraph follows from the first one. Let $Y$ and $Z$ be two dense subsets of $X$, we show that the intersection $Y\cap Z$ is dense in $X$. 

Towards a contradiction, suppose that $Y\cap Z$ is not dense, and let $x\in X\setminus cl(Y\cap Z)$. Let $W$ be an open neighborhood of $x$ that is disjoint from $Y\cap Z$. Then we have that $Y\cap W$ and $Z\cap W$ are disjoint and dense in $W$. In particular $Z\cap W\subseteq \partial (Y\cap W)$ and $Y\cap W\subseteq \partial (Z\cap W)$. By Lemma~\ref{lem:cons-noise0} we get that $Y$ and $Z$ are both non-constructible, contradicting that $\Mm$ is constructible. 
\end{proof} 

\begin{remark}
There is an obvious generalization of the notion of topological or uniform structure $\Mm$, where a definable basis for a topology or a uniformity is considered, not necessarily on $M$, but instead on some definable set $X$. Since this paper focuses on the investigation of open cores, we omit this perspective. Nevertheless, the reader will note that many results, including everything in this section, generalize to this setting.    
\end{remark}

\section{Strongly hereditarily Baire structures} \label{sec:BP}

%We now introduce a strong hereditary form of the Baire property which will hold in every $\NTP$ uniform structure. We isolate the property because it is upfront weaker than $\NTP$, and our main results rely only on it rather than the latter combinatorial restriction. Hence, it would be interesting to know which uniform structures satisfy it. In Section~\ref{sec:LO-structures}, we establish a different descriptive set-theoretic characterization of this property in structures expanding the real ordered group. Recall that a subset of a topological space is nowhere dense if its closure has empty interior. 

We now introduce a definable hereditary form of the Baire property and prove versions of Theorem~\ref{thm:main-intro} and Theorem~\ref{thm:f-intro} where we assume only this property in place of $\NTP$. In Section~\ref{sec:ntp}, we show that every $\NTP$ uniform structure has this property, yielding our main theorems. The reason we isolate this condition is that it is upfront weaker than $\NTP$. Hence, it would be interesting to know which uniform structures satisfy it (see Question~\ref{Q:BP}). In Remark~\ref{rem:simple-BP}, we isolate a slightly simpler more technical condition that is sufficient to derive our theorems and that will be used in Section~\ref{sec:LO-structures}. 

\begin{definition}
Let $\Mm$ be a topological structure. We say that $\Mm$ is \emph{Strongly Hereditarily Baire ($\BP$)} if, for every structure $\Nn\equiv \Mm$, every set $Z\subseteq N^m$, and every upward directed definable family of closed sets $\{ C_b : b\in B\}$ in $N^m$, if $C_b\cap Z$ has empty interior in $Z$ for every $b\in B$, then $\bigcup_{b\in B} C_b$ is nowhere dense in $Z$. 
\end{definition}

We leave it to the reader to check that the definition of $\BP$ is equivalent to asking that, for every set $Z\subseteq N^m$ and upward directed definable covering $\{C_b : b\in B\}$ of $Z$ by closed sets, there exists $b\in B$ such that $C_b$ has nonempty interior in $Z$. This definition is strictly stronger than satisfying the Strong Baire Category Theorem proved in~\cite{comp-set} for so-called type A expansions of $(\R,<,+)$. The notion was motivated by the analysis of $\NTP$ uniform structures laid out in Section~\ref{sec:ntp}, which resembles previous work of Walsberg~\cite{Wal22}.  
%Being $\BP$ should not be confused with topological properties such as being hereditarily Baire (every closed subspace is Baire).
% Complete metric spaces are hereditarily Baire. 

\begin{lemma}\label{lem:SHBP-Fsigma-cons}
Let $\Mm$ be a $\BP$ uniform structure. Then any $\dF$ set is constructible.  
\end{lemma}
\begin{proof}
Without loss of generality, we may assume that $\Mm$ is $\omega_1$-saturated. 
Suppose that there exists an $\dF$ set $X$ that is not constructible. Let $Z=\bigcap_{n<\omega} cl(\partial^{(n)} X)$. By Proposition~\ref{lem:cons-noise1}, the set $X$ is dense and codense in $Z$. The set $Z$ and the family $\{ C_b : b\in B\}$ witnessing that $X$ is $\dF$ witness the failure of $\BP$. 
\end{proof}   

By Lemma~\ref{lem:cons-dF}, the reverse implication to Lemma~\ref{lem:SHBP-Fsigma-cons} is also true. So, in an $\BP$ uniform structure, the constructible definable sets coincide with those that are $\dF$. Furthermore, they also coincide with the class of $\dG$ sets described in Remark~\ref{rem:dG}.  

Recall that the graph of a continuous partial function with closed domain is closed. As a consequence, if one such function is definable in a topological structure, it is also definable in the open core. In $\BP$ uniform structures, we may improve this fact as follows. We rely on the proof of the known fact that the set of points of continuity of a function between uniform spaces is $G_\delta$. 

%Closed graph of continuous function needs that the codomain be Hausdorff, which we are assuming

\begin{proposition}\label{prop:f_D}
Let $\Mm$ be a $\BP$ uniform structure, and let $f:X\subseteq M^n\rightarrow M^m$ be a definable function, where $X$ is a constructible set. Let $D=\{ x\in X : \text{ $f$ is continuous at $x$\}}$. Then $D$ and the restriction $f|_D$ are constructible (the latter meaning that the graph $\Gamma(f|_D)$ is constructible). In particular, $f|_D$ is definable in the open core $\Mo$.
\end{proposition}
\begin{proof}
We show that $D$ is constructible in the subspace topology in $X$. Since $X$ is constructible (in $M^n$), it follows that $D$ is also constructible. Furthermore, it is a standard exercise in continuity that $\Gamma(f|_D)$ is closed in $D\times M^m$. So, if $D$ is constructible, we have in particular that $D\times M^m$ is constructible, and so we conclude that $\Gamma(f|_D)$ is also constructible, completing the proof.

Let $\Bb_X$ and $\Bb_m$ be definable bases for the uniform structures on $X$ and $M^{m}$ respectively. For every $U\in \Bb_m$, let $X(U)=\{ x\in X : \exists V\in \Bb_X, \, (f(y),f(z))\in U \text{ for every } y,z\in V[x]\}$. Clearly every set $X(U)$ is open in $X$. Since $\Bb_m$ is downward directed, the definable family $\{X(U) : U\in \Bb_m\}$ is also downward directed. We sketch the details of the classical argument that $\bigcap_{U\in \Bb_m} X(U)=D$. The inclusion $\bigcap_{U\in \Bb_m} X(U) \subseteq D$ is obvious. Now let $x\in D$ and $U\in \Bb_m$. Let $W\in \Bb_m$ be such that $W\circ W \subseteq U$. Applying continuity, let $V\in \Bb_X$ be such that $V[x] \subseteq f^{-1}(W[f(x)])$. Then, for any $y, z \in V[x]$, we have that $(f(y),f(x))\in W$ and $(f(z),f(x))\in W$, and so $(f(y),f(z))\in U$, meaning that $x\in X(U)$. 

Hence the family $\{ X\setminus X(U) : U\in \Bb_m\}$ witnesses that $X\setminus D$ is $\dF$ in $X$ and so, by Lemma~\ref{lem:SHBP-Fsigma-cons}, the set $X\setminus D$ is constructible in $X$, and so is $D$, as desired. 
\end{proof}

By Proposition~\ref{prop:BCT} in the next section, the above proposition also applies to all $\NTP$ uniform structures. 

\subsection{Definably $\sigma$-compact structures}

We now begin exploring our compactness hypothesis. Recall that a topological space is $\sigma$-compact if it has a countable cover by compact subsets. We introduce a definable analogue of this property. Much like with the notion of $\dF$, our approach is that of substituting countability by ``definable upward directedness". We will later observe that our definition is always satisfied by expansions of $(\R,<)$ and $(\Qp,+,\cdot)$, as well as by definably complete ordered structures.

\begin{definition}\label{def:compact}
A definable set $X$ in a topological structure $\Mm$ is \emph{definably compact} if, for every definable downward directed family $\{C_b :b\in B\}$ of nonempty closed subsets of $X$ (in the subspace topology), it holds that $\bigcap_{b\in B} C_b \neq \emptyset$. 
We say that $X$ is \emph{definably $\sigma$-compact} if it has a covering given by an upward directed definable family of definably compact subsets. We say that $\Mm$ is definably $\sigma$-compact if $M$ is. 
\end{definition}

Our notion of definable compactness has already been studied in various model-theoretic contexts~\cite{johnson14, and-defcomp, JY-padic-comp}. Observe that both notions, definable compactness and definable $\sigma$-compactness, are properties of the underlying theory of the structure. 

In the next fact we compile the properties of definable compactness that we will need. 
\begin{fact}[\cite{johnson14}] \label{fct:comp-basic}
Let $\Mm$ be a topological structure. Let $X$, $Y$ and $C$ be definable sets, with $C\subseteq X$.
\begin{enumerate}[(1)]
    \item If $X$ is definably compact and $C$ is closed in $X$, then $C$ is definably compact. 
    \item If $C$ is definably compact, then $C$ is closed\footnote{For this we require that the topology be Hausdorff, which we are assuming throughout.}. 
    \item If $f:X\rightarrow Y$ is a definable continuous function and $C$ is definably compact (respectively $\sigma$-compact), then $f(C)$ is definably compact (respectively $\sigma$-compact). 
    \item If $X$ and $Y$ are definably compact (respectively $\sigma$-compact), then $X\times Y$ is definably compact (respectively $\sigma$-compact).
\end{enumerate}
\end{fact}

\begin{lemma}\label{lem:compdF}
Let $\Mm$ be a definably $\sigma$-compact uniform structure. Every $\dF$ set is definably $\sigma$-compact.
\end{lemma}
\begin{proof}
Let $X\subseteq M^n$ be a $\dF$ set, witnessed by an upward directed definable family of closed sets $\{ C_b : b\in B\}$. By Fact~\ref{fct:comp-basic} (4),
let $\{ K_d: d\in D\}$ be an upward directed definable covering of $M^n$ by definably compact sets. By Fact~\ref{fct:comp-basic} (1), the family $\{ C_b \cap K_d : b\in B,\, d\in D\}$ witnesses that $X$ is definably $\sigma$-compact. 
\end{proof}

The following lemma distills our main use of definable compactness. 

\begin{lemma}\label{lem:image-dF}
Let $\Mm$ be a definably $\sigma$-compact uniform structure. Let $X$ be $\dF$ and $f:X\rightarrow M^n$ be a definable continuous function. Then $f(X)$ is $\dF$. 
\end{lemma}
\begin{proof}
Applying Lemma~\ref{lem:compdF}, let $\{ C_b: b\in B\}$ be a definable upward directed covering of $X$ by definably compact sets. By Fact~\ref{fct:comp-basic} (2) and (3), the definable family $\{ f(C_b) : b\in B\}$ witnesses that $f(X)$ is $\dF$. 
\end{proof}

We are now ready to prove the main part of Theorem~\ref{thm:main-intro}, with the $\BP$ assumption in place of $\NTP$. Relying on the work above, we use an approach similar to the one used in proving the main theorems in \cite{MS99} and in \cite{DMS10}. 
%The remainder of the proof is similar to \cite[2.7]{DMS10}.

\begin{theorem}\label{thm:main-BP}
Let $\Mm$ be a $\BP$ uniform structure. If $\Mm$ is definably $\sigma$-compact, then its open core $\Mm^o$ is constructible. 
\end{theorem}
\begin{proof}
In $\Mo$ the quantifier-free formulas are exactly those defining constructible sets. It is a standard fact that, in order to prove quantifier elimination, it suffices to eliminate one existential quantifier. Hence, it suffices to show that the projection of a definable constructible set is constructible. We fix a definable constructible set $X \subseteq M^{n}$, and a projection $\pi:M^n\rightarrow M^m$ with $n\geq m>0$.  
By Lemma~\ref{lem:cons-dF}, $X$ is $\dF$. By Lemma~\ref{lem:image-dF}, the set $\pi(X)$ is $\dF$. Finally, by Lemma~\ref{lem:SHBP-Fsigma-cons}, the projection $\pi(X)$ is constructible. 
\end{proof}

The next corollary highlights how our hypothesis of definable $\sigma$-compactness generalizes the use of classical $\sigma$-compactness. 

\begin{corollary}\label{cor:real-sigma-comp}
Let $\Mm$ be a uniform structure. Suppose that there exists a definable upward directed covering of $M$ by compact sets. Then the open core of every $\BP$ expansion of $\Mm$ is constructible. In particular, the open core of every $\BP$ expansion of the following structures is constructible: 
\begin{enumerate}
    \item The ordered group of reals $(\R, +, <)$.
    \item The field $(\Qp, +, \cdot)$ of $p$-adic numbers, for every prime $p$. 
\end{enumerate}
\end{corollary}
\begin{proof}
Topological compactness is trivially maintained after passing to any expansion. Furthermore, note that any compact set is definably compact. Hence any expansion of $\Mm$ is definably $\sigma$-compact, witnessed by the same upward directed covering, and the result follows from Theorem~\ref{thm:main-BP}.
\end{proof}

Constructibility has strong topological consequences in expansions of ordered groups. Some of these will be presented in Section~\ref{sec:LO-structures}, with a focus on expansions of $(\R,<,+)$. In this setting, a consequence is that graphs of definable functions $f:U\subseteq \R^n \rightarrow \R^m$ are $F_\sigma$, and so one may show that they are Baire class one (see~\cite{kechris-book}), which implies that the set of points of continuity of $f$ is comeager ---in particular dense--- in $U$ constructible.
% From Kechris Book (Thm 24.14): Let X,Y be metrizable, with Y separable, and $f:X \rightarrow Y$ be of Baire class 1. Then the set of points of continuity of $f$ is a comeager $G_\delta$ set. 
Definable versions of this result have been proved, occasionally under the name of generic continuity, in various tame settings~\cite{Wal22, def-baire, HW-tetrachotomy}.
%G_delta sets subsets of R^n are Baire, so constructible sets are Baire, namely the BCT can be applied to them. (A subset of R^n is G_delta iff it is completely metrizable.)
Nevertheless, if one considers, not just constructibility, but being $\BP$, then generic continuity can be sharpened: we show that $f$ is also (finitely) piecewise continuous. This improves a more general theorem of Jayne and Rogers~\cite{Jayne-Rogers} applied to this setting, which shows that $f$ must be countably piecewise continuous.  

%For proving the "generic continuity part", the following proof uses ideas from Lemma 3.11 in ``Definably complete Baire structures and Pfaffian closure", where it is proved for comeager instead of open dense, in def. complete expansions of fields with the assumption that the graph of $f$ is an $\Ff_\sigma$ set.

\begin{theorem} \label{thm:f-BP}
Let $\Mm$ be a definably $\sigma$-compact $\BP$ uniform structure. Let $f:X\rightarrow Y$ be a function definable in the open core. Let $X_0=X$ and, for every $n$, let $X_{n+1}$ be the closure in $X_n$ of the set of points of discontinuity of $f|_{X_n}$. Then:
\begin{enumerate}
     \item $X_j$ has empty interior in $X_i$ for every $i<j$.
     \item There exists $m$ such that $X_m=\emptyset$.   
\end{enumerate}
In particular, $f$ is both generically and piecewise continuous. 
\end{theorem}
\begin{proof}
We show that $X_1$ is nowhere dense in $X=X_0$. Since, for every $n$, the function $f|_{X_n}$ is clearly definable in the open core, this yields that $X_j$ is nowhere dense in $X_i$ for every $i<j$. 

We first show that, if the graph $\Gamma(f)$ of $f$ is definably compact, then $f$ is continuous. Let $\BU_X$ and $\BU_Y$ denote definable bases for the uniformities in $X$ and $Y$ respectively. Towards a contradiction, suppose that $f$ is not continuous at a point $x\in X$. Then there exists some $V\in \BU_Y$ satisfying that, for every $U\in \BU_X$, the set $\Gamma_U=\{ (y,f(y))\in \Gamma(f) : (x,y)\in U \text{ and } (f(x),f(y))\notin V\}$ is nonempty. Consider the downward directed definable family of closed sets $\Dd=\{ cl(\Gamma_U) : U\in \BU_X\}$. By definable compactness, there exists $(z,f(z))\in \bigcap \Dd$. By Hausdorffness, note that $z=x$. On the other hand, clearly $(X\times V[f(x)]) \cap \Gamma_U = \emptyset$ for every $U\in \BU_X$, contradiction.

%Now, for the general case, let $Z\subseteq X$ denote the set of points where $f$ is discontinuous. Observe that this set is definable in the open core and so, by Theorem~\ref{thm:main-tp2}, it is constructible. We will show that $Z$ has empty interior in $X$. By Lemma~\ref{lem:cons-noise0}, its closure $cl(Z)$ has also empty interior in $X$, and the theorem follows. 

%Towards a contradiction, suppose that $Z$ has nonempty interior in $X$. By restricting $f$ to the interior of $Z$ if necessary, we may assume that $f$ is nowhere continuous. Now, by Theorem~\ref{thm:main-tp2}, the graph $\Gamma(f)$ of $f$ is constructible. By Lemma~\ref{lem:cons-dF}, $\Gamma(f)$ is $\dF$. By Lemma~\ref{lem:compdF}, there exists an upward directed definable family $\{C_b : b\in B \}$ of definably compact sets which covers $\Gamma(f)$. For each $b\in B$, let $\pi(C_b)$ be the projection of $C_b$ into $X$. By Fact~\ref{fct:comp-basic}, each set $\pi(C_b)$ is closed. By Proposition~\ref{prop:BCT}, there exists $b\in B$ such that $\pi(C_b)$ has nonempty interior in $X$. Let $B=\pi(C_b)$ and $x$ be a point in the interior of $B$. Then $f|_B$ has a definably compact graph and is discontinuous at $x$, contradiction.  

Now let $Z\subseteq X$ denote the set of points where $f$ is discontinuous. Observe that this set is definable in the open core and so, by Theorem~\ref{thm:main-BP}, it is constructible. We show that $Z$ has empty interior in $X$. By Lemma~\ref{lem:cons-noise0}, it then follows that $X_1=cl(Z)\cap X$ has also empty interior in $X$. 

Towards a contradiction, suppose that $Z$ has nonempty interior in $X$. By restricting $f$ to the interior of $Z$ if necessary, we may assume that $f$ is nowhere continuous. Now, by Theorem~\ref{thm:main-BP}, the graph $\Gamma(f)$ of $f$ is constructible. By Lemma~\ref{lem:cons-dF}, $\Gamma(f)$ is $\dF$. By Lemma~\ref{lem:compdF}, there exists an upward directed definable family $\{C_b : b\in B \}$ of definably compact sets which covers $\Gamma(f)$. For each $b\in B$, let $\pi(C_b)$ be the projection of $C_b$ into $X$. By Fact~\ref{fct:comp-basic}, each set $\pi(C_b)$ is closed. Since $\Mm$ is $\BP$, there exists $b\in B$ such that $\pi(C_b)$ has nonempty interior in $X$. Let $x$ be a point in the interior of $\pi(C_b)$. Then $f|_{\pi(C_b)}$ has a definably compact graph and is discontinuous at $x$, contradiction.

We now prove $(2)$ by assuming, towards a contradiction, that $X_n\neq \emptyset$ for every $n$. Without loss of generality, we may assume that $\Mm$ is $\omega_1$-saturated. Let $K$ denote the (nonempty) type-definable set $K=\bigcap_{n<\omega} X_n$. Following the notation in the proof of statement (1), let $\{C_b : b\in B \}$ be an upward directed definable covering of $\Gamma(f)$ by definably compact sets, and $\{\pi(C_b) : b\in B \}$ be the corresponding family of (definably compact and closed) projections into $X$. Since $\Mm$ is $\BP$, there exists some $b\in B$ such that $K\cap \pi(C_b)$ has nonempty interior in $K$. Let $x\in K$ and $U\in \BU_X$ be such that $K\cap U[x] \subseteq \pi(C_b)$. By saturation, there must exist some $n$ such that $X_n\cap U[x] \subseteq \pi(C_b)$. Now, by Fact~\ref{fct:comp-basic}, note that $\Gamma(f)\cap C_b \cap (X_n\times Y)$ is definably compact, and so the restriction of $f$ to $X_{n} \cap \pi(C_b)$ is continuous. Hence $f|_{X_{n}}$ is continuous on $X_{n}\cap U[x]$, however this contradicts the fact that $x\in X_{n+1}$. 
\end{proof}

\begin{remark}\label{rem:simple-BP}
Close analysis of the proofs of Theorems~\ref{thm:main-BP} (through Lemma~\ref{lem:SHBP-Fsigma-cons}) and~\ref{thm:f-BP} yields that we do not require the full property of being $\BP$ in order to prove these results in a definably $\sigma$-compact uniform structure $\Mm$. Instead, we only need that there exists an $\omega_1$-saturated elementary extension $\Nn$ of $\Mm$ with the following property. If $X$ is a $\dF$ set in $\Mm$, whose interpretation in $\Nn$ we denote by $X^{\Nn}$, and $Z\subseteq X^{\Nn}$ is a type-definable set in $\Nn$ defined by a countable type over $M$, then there exists a definably compact subset $C\subseteq X^{\Nn}$ in $\Nn$ such that $C\cap Z$ has nonempty interior in $Z$.  

Suppose that $\Mm$ expands an ordered abelian group. A definable set $X$ is $D_\Sigma$ in the sense of~\cite{DMS10} if there exists a covering of $X$ by a definable family $\{C_{r,s} : (r,s)\in M^{>0}\times M^{>0}\}$ of closed and bounded sets satisfying that $C_{r,s}\subseteq C_{r',s'}$ whenever $r\leq r'$ and $s\geq s'$. Every definable constructible set is $D_\Sigma$~\cite[1.10]{DMS10}. As a consequence, under this assumption on $\Mm$, in the above paragraph one may consider $X$ to be $D_\Sigma$ rather than the more general notion of $\dF$.
\end{remark}
%Above we say the $C$ are definably compact instead of closed in order to allow the proof of generic piecewise continuity (Theorem~\ref{thm:f-BP}). If we drop the compactness assumption then we can probably prove a similar theorem to~\ref{thm:Baire-class-one} in expansion of $(\R,<,+)$ (because families witnessing $\dF$ are indexed by $\R$ and increasing at t decreases). 

We will use the above remark to prove our descriptive set-theoretic version of Theorems~\ref{thm:main-intro} and~\ref{thm:f-intro} in the Euclidean setting (Theorem~\ref{thm:Baire-class-one}).

\section{\NTP{} uniform structures} \label{sec:ntp}

In this section we prove the main part of Theorem~\ref{thm:main-intro}, as well as Theorem~\ref{thm:f-intro}, by showing that $\NTP$ uniform structures are $\BP$, and then applying the work in Section~\ref{sec:BP}. We begin with a major technical lemma of this paper. Compare to~\cite[Theorem 4.1]{Wal22}, where Walsberg proves roughly the same result but assuming that $Z$ is definable. Having the lemma for an arbitrary $Z$ is crucial in our proofs. 
%on constructible sets (theorem and prop) and on Cantor-Bendixson rank (result). 

In the proof of the next lemma we generalize our $U$-ball notation as follows. Given a uniform structure $\Mm$, a set $X\subseteq M^n$, and basic entourages $U,V$ for the uniformity on $M^d$, let $(U\setminus V)[X]=U[X]\setminus V[X]$. 

\begin{lemma}\label{lem:noise-NTP2}
Let $\Mm$ be a uniform structure and let $Z\subseteq M^d$ be a nonempty subset. Let $\Cc$ be a definable family of closed subsets of $M^d$ with the following property. For every finite collection $W_1,\ldots W_k$ of nonempty subsets of $Z$ that are open in the subspace topology, there exists some $C\in \Cc$ such that
\begin{enumerate} 
    \item $C \cap W_i\neq \emptyset$ for each $i\leq k$,
    \item $C$ is codense in $Z$. 
\end{enumerate}
Then $Z$ has \TP.
\end{lemma}
\begin{proof}
Let $\BU$ be a definable basis for the uniformity on $M^d$. 
For each $n$ and $m$, we define recursively an array of tuples $(C_{i}, U_{i,j}) \in \Cc\times\BU$, for $i\leq n$ and $j\leq m+1$, such that the family of sets $(U_{i,j}\setminus U_{i,j+1})[C_i] = U_{i,j}[C_i]\setminus U_{i,j+1}[C_i]$ witnesses that $Z$ has \TP. For a fixed $i$, we will pick the family $\{ U_{i,j} : j\leq m+1\}$ to be decreasing, which ensures that the sets $\{(U_{i,j}\setminus U_{i,j+1})[C_i] : j \leq m\}$ are pairwise disjoint, and so it will suffice to check the consistency condition. We will also pick the array such that, for every function $\eta:\{1,\ldots, n\}\rightarrow \{1,\ldots, m\}$, the intersection $\bigcap_{i=1}^{n} (U_{i,\eta(i)}\setminus U_{i,\eta(i)+1})[C_i]$ not only is nonempty, but has in fact nonempty interior in $Z$. 

Now let $C_1\in \Cc$ be any set such that $C_1\cap Z \neq \emptyset$ and $C_1$ is codense in $Z$, which exists by applying the lemma assumptions with $W_1=Z$. Let $U_{1,1}\in \BU$ be any basic entourage. Since $C_1$ has empty interior in $Z$, there exists some $x\in Z \cap U_{1,1}[C_1]\setminus C_1$. Since $C_1$ is closed, there exists some $V\in \BU$ such that  $V[x] \cap C_1 = \emptyset$. Using properties of entourages, if we pick $U_{1,2}\in \BU$ with $U_{1,2} \circ U_{1,2}\subseteq V$, it is easy to see that $U_{1,2}[x] \cap U_{1,2}[C_1] = \emptyset$. In particular we get that $x$ is in the interior in $Z$ of the set $(U_{1,1}\setminus U_{1,2})[C_1]$. We may furthermore clearly pick $U_{1,2}$ small enough to satisfy $U_{1,2}\subseteq U_{1,1}$. Now, if $m>1$, we may recursively continue the process and pick, for any $U_{1,j} \in \BU$, some set $U_{1,j+1}\in \BU$ with $U_{1,j+1}\subseteq U_{1,j}$ satisfying that $(U_{1,j}\setminus U_{1,j+1})[C_1]$ has nonempty interior in $Z$. 

Now suppose that $n>1$. We assume that we have tuples $(C_{i}, U_{i,j})$, for $i<n$ and $j\leq m+1$, satisfying that, for every $\eta: \{1,\ldots, n-1\} \rightarrow \{1,\ldots m\}$, the interior in $Z$ of the intersection 
$
\bigcap_{i=1}^{n-1} (U_{i,\eta(i)}\setminus U_{i,\eta(i)+1})[C_i] 
$
is nonempty. We denote this interior by $W_\eta$. 

By the lemma assumptions we may pick $C_{n}\in \Cc$ such that $C_n\cap W_{\eta} \neq \emptyset$ for every $\eta$ and $C_n$ is codense in $Z$. Let $U_{n,1}\in \BU$ be any basic entourage. Now, in analogy to the case $n=1$, using both that $C_n$ is closed and that it has empty interior in $Z$ ---and in particular in every set $W_\eta$---, we may recursively pick, for $U_{n,j}\in \BU$, a subset $U_{n,j+1}\in \BU$ satisfying that, for every $\eta: \{1,\ldots, n-1\} \rightarrow \{1,\ldots m\}$, the intersection $W_{\eta} \cap (U_{n,j}\setminus U_{n,j+1})[C_n]$ has nonempty interior in $Z$. 
\end{proof}

Walsberg proves \cite[Theorem 4.1]{Wal22} assuming only that $\BU$ is a subfamily of some definable family of sets. Note that the above proof remains valid with this weaker assumption. 

%NB: If $Z$ is definable then we only need that the intersections are closed since we can pass from X to Y. 

%Above lemma can perhaps be proved with constructible instead of closed, as long as there is a uniform bound the in description of the constructible sets.  

We may now establish the connection between $\NTP$ and $\BP$. Compare to~\cite[Corollary 4.2]{Wal22}.

\begin{proposition}[$\NTP$ Baire category theorem]\label{prop:BCT}
Let $\Mm$ be an \NTP{} uniform structure. Then $\Mm$ is strongly hereditarily Baire ($\BP$).
\end{proposition}
\begin{proof}
Let $X\subseteq M^n$ be a set and $\Cc$ be an upward directed definable family of closed sets that are each nowhere dense in $X$, we must show that $\bigcup \Cc$ is nowhere dense in $X$. Suppose that $\bigcup \Cc$ is dense in some nonempty open subset $X'\subseteq X$. We may assume that $X'=X$. Using that each $C\in \Cc$ is nowhere dense in $X$, we apply Lemma~\ref{lem:noise-NTP2} to derive that $\Mm$ has $\TP$. Since $\NTP$ is a property of the underlying theory we conclude that $\Mm$ is $\BP$.  
\end{proof}

\begin{proof}[Proof of Theorems~\ref{thm:main-intro} and~\ref{thm:f-intro}]
Theorem~\ref{thm:main-intro} is an immediate consequence of Proposition~\ref{prop:BCT}, Theorem~\ref{thm:main-BP}, and Corollary~\ref{cor:real-sigma-comp}, with the sole remaining task of showing that every definably complete expansion of an ordered group is definably $\sigma$-compact, which is done in Lemma~\ref{lem:intervals-comp}. Theorem~\ref{thm:f-intro} is an immediate consequence of Proposition~\ref{prop:BCT} and Theorem~\ref{thm:f-BP}.
\end{proof}

\begin{remark}\label{rem:TP2-cons}
\begin{enumerate}
\item There are constructible structures over $(\R,<,+)$ that are not combinatorially tame. Let $(\R,+,<, \Z)^\#$ denote the expansion of $(\R,<,+)$ by every subset of $\Z^n$ for every $n$. By the example in~\cite[page 56]{FM01-sparse}, this structure is locally o-minimal. Using Theorem 3.2 in~\cite{miller05-tame}, it is easy to see that it is constructible. On the other hand, $(\R,+,<, \Z)^\#$ clearly fails to satisfy any Shelah-style combinatorial tameness condition, since, for example, it defines every finite subset of $\Z^n$ uniformly.  
%Let $B\subseteq \N^2$ be a set such that, for any finite subset $F\subseteq \N$, there is $a\in \N$ such that $B_a=\{ b\in \N : (a,b)\in B\}=F$.
%There are constructible structures over $(\R,<,+)$ that are not combinatorially tame. By~\cite{FM01-sparse}, the structure $(\overline{\R}, 2^\Z)^\#$ is d-minimal, in particular it is constructible. On the other hand, it defines an isomorphic copy of $(\Z,+,\cdot)$ as well as every finite subset of $2^\Z$ uniformly. In particular its theory is undecidable and it does not satisfy any Shelah-style combinatorial tameness condition

\item The condition of definable $\sigma$-compactness in Theorem~\ref{thm:main-intro} is not necessary. The structure $(\Q,+,<,I)$, where $I=(-\pi,\pi)\cap \Q$, is both $\NIP$ and constructible (see Theorem~\ref{thm:nip-intro}). However, it is not definably $\sigma$-compact. We sketch the argument for this. It is easy to see that every interval in $(\Q,+,<,I)$ is not definably complete, and in particular not definably compact. So every definably compact subset of $\Q$ has empty interior. By Proposition~\ref{prop:BCT}, it follows that $(\Q,+,<,I)$ does not admit an upward directed definable covering of $\Q$ by definably compact sets. 

\item On the other hand, Theorem~\ref{thm:main-intro} is false if we simply drop either the $\NTP$ or the definable $\sigma$-compactness assumption. The latter is witnessed by the following example brought to the author's attention by Johnson and Goodrick\footnote{A similar example is discussed in~\cite{goo10}}. The ordered group $(\Z_{(2)},+,<)$ of integers localized at $2$ is dp-minimal (see Proposition 5.1 in \cite{JSW-valued-fields}) ---hence $\NIP$--- and clearly interdefinable with its open core, however the set $2\Z_{(2)}$ of elements divisible by $2$ is dense and codense in $\Z_{(2)}$, and so it is not constructible. 
\end{enumerate}
\end{remark}

By Proposition~\ref{prop:BCT} and Lemma~\ref{lem:SHBP-Fsigma-cons}, every $\dF$ set in an $\NTP$ uniform structure is constructible. In fact, this result remains true if one only assumes $\NTP$ localized in the usual way to $X$. Given an expansion of $(\R,<,+)$, a dense $\omega$-order is a pair $(X,\prec)$ where $X$ is a somewhere dense subset of $\R$ and $\prec$ is a total order on $X$ of order type $\omega$. Observe that a definable dense $\omega$-order $X$ is always $\dF$, because $X=\bigcup_{b\in X} \{ a\in X : a \prec b\}$, but it is not constructible, by density and codensity of $X$. This example is in fact a particularly pathological case of a $\dF$ non-constructible set, which yields combinatorial wildness beyond $\TP$ \cite{HW18}. 

\begin{remark}\label{rem:T-open-core}
Being constructible is maintained under elementary equivalence (see Remark~\ref{rem:extensions}). On the other hand, given a topological structure $\Mm$ with constructible open core, it is not clear whether any structure elementarily equivalent to $\Mm$ also has constructible open core. Since being $\NTP$ and definably $\sigma$-compact are also maintained under elementary equivalence, Theorem~\ref{thm:main-intro} provides a positive answer to this question for uniform structures with these properties. In Proposition~\ref{prop:cons-OC} we show that having constructible open core is also maintained under elementary equivalence for the class of strong uniform structures.
\end{remark}

In~\cite{Hie10} Hieronymi proves that, given a closed discrete set $D\subseteq \R$ and function $f:D^n\rightarrow \R$ whose image $f(D^n)$ is somewhere dense, the structure $(\R,+,\cdot, f)$ defines $\Z$. In particular, one such structure cannot satisfy Shelah-style combinatorial tameness conditions such as $\NTP$. This was partially generalized in~\cite[Theorem C]{HW18}, where it is shown that, if an expansion $\Mm$ of $(\R,<,+)$ defines an $\R$-valued function on a discrete set whose image is somewhere dense, then $\Mm$ defines a copy of the two-sorted structure $(\mathcal{P}(\N), \N, \in, +1)$. The next corollary strengthens the $\NTP$ versions of these results. 

\begin{corollary}\label{cor:disc-non-cons}
Let $\Mm$ be an \NTP{} definably $\sigma$-compact uniform structure. Let $D$ be a definable discrete set and $f:D\rightarrow M^n$ be a definable function. Then $f(D)$ is constructible. 
\end{corollary}
\begin{proof}
Since $D$ is discrete, it is constructible, and so by Lemma~\ref{lem:cons-dF} it is $\dF$. Since any function on a discrete set is continuous, by Lemma~\ref{lem:image-dF} the image $f(D)$ is $\dF$. By Proposition~\ref{prop:BCT} and Lemma~\ref{lem:SHBP-Fsigma-cons} we conclude that $f(D)$ is constructible. 
\end{proof}

We leave it to the reader to check, using the approach in the proof of Corollary~\ref{cor:disc-non-cons}, that an $\NTP$ definably $\sigma$-compact uniform structure cannot define a continuous surjection from a constructible to a non-constructible set. On the other hand, there can exist a definable bijection between a constructible and a non-constructible set. For example, in $(\R,<,+,\Q)$, consider the bijection $f:[0,1]\rightarrow ([0,1] \cap \Q) \cup ([1,2]\setminus \Q)$ given piecewise by $f(x)=x$ on $[0,1]\cap \Q$ and $f(x)=x+1$ on $[0,1]\setminus \Q$. 

%\begin{remark}
%Observe that many all of these results only require that the open core be NTP2.
%\end{remark}

The main open question of this paper is whether Theorems~\ref{thm:main-intro} and~\ref{thm:f-intro} can be established under combinatorial conditions weaker than $\NTP$. We formulate a precise question under what is arguably a maximally weak combinatorial tameness assumption, namely the non-definability of a copy of the monadic second-order theory of $\N$. 

\begin{question}\label{Q:BP}
Let $\Rr$ be an expansion of $(\R,<,+)$ that does not define an isomorphic copy of the two sorted structure $(\mathcal{P}(\N), \N, \in)$. Is $\Rr$ strongly hereditarily Baire? In particular, is $\Ro$ constructible?
\end{question}
%See Conjecture 2 in~\cite{wal22-preprint}.

\section{$\NIP$ uniform structures} \label{sec:nip}

In this section we focus on the subclass of $\NTP$ uniform structures \emph{without the independence property ($\NIP$)}. An example of one such structure that is not definably $\sigma$-compact would be $\Rr_\pi=(\Ralg, <, +,I)$, where $\Ralg$ denotes the real algebraic numbers and $I=\Ralg \cap (-\pi, \pi)$ (see Remark~\ref{rem:TP2-cons} (2)). Using the known fact that $\Rr_\pi$ is dp-minimal, one may nevertheless derive from~\cite{sim-wals19} that it is constructible.
%By (ref), this structure has (non-valuational) weakly o-minimal theory.
In this section we prove that a general class of structures, which includes $\Rr_\pi$, are constructible. By Proposition 2.1 in \cite{dolich-note}, the structure $\Rr_\pi$ contains projections of closed and bounded definable sets that are not closed, which complicates the implementation of the ideas used in the proof of Theorem~\ref{thm:main-intro} in this setting. Instead, we rely on $\NIP$ \emph{honest definitions}.

Given a structure $\Mm$, an \emph{externally definable} subset of $M^n$ is a set of the form $\varphi(M,b)$, where $\varphi(x,y)$ is a formula with $|x|=n$, and $b\in M^{|y|}$ is a tuple of parameters in some elementary extension $\Nn$ of $\Mm$. (For example, the set $I=\Ralg \cap (-\pi, \pi)$ is externally definable in $(\Ralg, <, +)$.)
The \emph{Shelah expansion of a structure $\Mm$}, denoted $\Msh$, is the expansion of $\Mm$ in the language where a predicate is added for every externally definable set.

\begin{fact}[\cite{simon:guide-NIP}, honest definitions] \label{honest-definitions}
Let $\Mm$ be an $\NIP$ structure, and let $X$ be an externally definable subset of $M^n$. There exists a formula $\varphi(x,y)$, with $|x|=n$, such that, for every finite set $F\subseteq X$, there exists some $b\in M^{|y|}$ satisfying that $F\subseteq \varphi(M,b) \subseteq X$. 
\end{fact}
%Honest definitions = Remark 3.14 in Simon's book. 

Recall that Lemma~\ref{lem:SHBP-Fsigma-cons} and Proposition~\ref{prop:BCT} yield that every $\dF$ set in an $\NTP$ uniform structures is constructible. We now need a slight modification of this result, with an analogous proof.

\begin{lemma}\label{lem:NIP-Fsigma}
Let $\Mm$ be an $\NTP$ uniform structure and $X\subseteq M^n$ be a definable set. Suppose that there exists a definable family $\{ X_b : b\in B\}$ of subsets of $M^n$ satisfying that, for every finite set $F\subseteq X$, there is some $b\in B$ such that $X_b$ is closed and moreover $F\subseteq X_b \subseteq X$. Then $X$ is constructible. 
\end{lemma}
\begin{proof}
Towards a contradiction, suppose that $X$ is not constructible. Without loss of generality, we may assume that $\Mm$ is $\omega_1$-saturated. Let $Z=\bigcap_{n<\omega} cl(\partial^{(n)} X)$. By Proposition~\ref{lem:cons-noise1}, the set $X$ is dense and codense in $Z$. But then, applying Lemma~\ref{lem:noise-NTP2} to $Z$ and the family $\Cc=\{ X_b : b\in B\}$, we derive that $\Mm$ has $\TP$. 
\end{proof}

We now observe that, in constructible uniform structures, the property of being $\dF$ is witnessed uniformly in families. Note that the same argument holds in structures where every definable set is $\dF$.
%Recall that, since a fiber of a constructible set is constructible, to have that a topological structure is constructible it suffices that $0$-definable sets are. In particular, being constructible is an elementary property. In the following remark we observe that, in constructible structures, the property of being $\dF$ is witnessed uniformly in families. We note this through a direct argument rather than by using model-theoretic compactness. We leave it to the reader to check that the same holds in structures where every definable set is $\dF$.

\begin{remark}\label{rem:dF-uniform}
Let $\Mm$ be a constructible uniform structure and let $\{ X_b : b\in B \subseteq M^n\}$ be a definable family of subsets of $M^m$. Let $X\subseteq M^{m+n}$ be the definable set such that, for each $b\in B$, the set $X_b$ is the corresponding fiber of $X$. Since $X$ is constructible, it is $\dF$ (Lemma~\ref{lem:cons-dF}), witnessed by some definable family $\{ C_d : d\in D\subseteq M^k\}$. By letting $C_d=\emptyset$ for all $d\notin D$ if necessary, we may assume that $D=M^{k}$. For each $d\in M^k$ and $b\in B$, let $C_{d,b}=(C_d)_b=\{ a\in M^n : (a,b)\in C_d\}$ denote the corresponding fiber of $C_d$. It is easy to see that, for every $b\in B$, the definable family $\{ C_{d,b} : d\in M^k\}$ witnesses that $X_b$ is $\dF$. 
\end{remark}

We may now prove Theorem~\ref{thm:nip-intro}.  

\begin{proof}[Proof of Theorem~\ref{thm:nip-intro}]
Let $X\subseteq M^n$ be a set definable in $\Msh$. By quantifier elimination of the Shelah expansion (see e.g. Proposition 3.24 in~\cite{simon:guide-NIP}) the set $X$ is externally definable. Let $\varphi(x,y)$ be a formula as given by Fact~\ref{honest-definitions}. Since $\Mm$ is constructible, by Remark~\ref{rem:dF-uniform} there exists another formula $\psi(x,y,z)$ such that, for every $b\in M^{|y|}$, the family $\{\psi(M,b,d) : d\in M^{|z|}\}$ witnesses that $\varphi(M,b)$ is $\dF$. By definition of $\varphi(x,y)$ it follows that, for every finite set $F\subseteq X$, there are parameters $(b,d)\in M^{|y|+|z|}$ such that the set $\psi(M,b,d)$ is closed and $F\subseteq \psi(M,b,d)\subseteq X$. Finally, since $\Msh$ is $\NIP$ (see Corollary 3.24 in \cite{simon:guide-NIP}) ---and so in particular $\NTP$---, we derive from Lemma~\ref{lem:NIP-Fsigma} that $X$ is constructible. 
\end{proof}

%The above theorem applies, for example, to weakly o-minimal structures generated by expanding o-minimal structures by externally definable sets, such as the example $\Rr_\pi$ described at the start of the section.  
Applying Theorem~\ref{thm:main-intro}, Theorem~\ref{thm:nip-intro} yields that the Shelah expansion of the open core of any $\NIP$ definably $\sigma$-compact uniform structure is constructible. We end the section by listing specific classes of $\NIP$ constructible structures. The constructibility of their Shelah expansion is then derived either from Theorem~\ref{thm:nip-intro} or, in some cases, directly from the work in~\cite{sim-wals19}. Case (6) is improved in Proposition~\ref{prop:shelah-d-min}.

\begin{corollary}\label{cor:shelah-exp}
Let $\Mm$ be any of the following structures. 
\begin{enumerate}
    \item An o-minimal expansion of an ordered group.  
    \item A P-minimal expansion of a p-adically closed field. 
    \item A C-minimal expansion of a C-group. 
    \item A dp-minimal expansion of a divisible ordered abelian group.
    \item A non strongly minimal dp-minimal expansion of a field with the Johnson topology \cite{johnson-topology}.
    \item An $\NIP$ d-minimal expansion of $(\R,<,+)$.
\end{enumerate}
Every expansion of $\Mm$ by externally definable sets is constructible. 
\end{corollary}
\begin{proof}
$(1)$ is a special case of $(4)$ and $(2)$ a special case of $(5)$. All of $(1)-(5)$ fall under the framework investigated in \cite{sim-wals19}, where it is shown that every such structure $\Mm$ admits a dimension $\dim X$ on definable sets $X$ which, assuming $X\neq \emptyset$, satisfies:  
\begin{enumerate}[(i)]
    \item $\dim X \geq 0$ and, moreover, $\dim X=0$ iff $X$ is finite,
    \item $\dim \partial X < \dim X$. 
    %Proposition 4.3. in \cite{sim-wals19}.
\end{enumerate}
Hence, since the topologies are Hausdorff, there exists some $n$ such that $\partn{n}X = \emptyset$, and so $\Mm$ is constructible (Fact~\ref{fct:constructible0}). The constructibility of every expansion of $\Mm$ by externally definable sets follows from Theorem~\ref{thm:nip-intro} or, in case $(1)$, $(2)$, $(4)$ and $(5)$, simply from the fact that the Shelah expansion remains dp-minimal. 

Let $\Rr$ be a d-minimal (see Section~\ref{sec:d-minimal}) expansion of $(\R,<,+)$ and $\Mm\equiv \Rr$. It is easy to see that, for every set $X$ definable in $\Mm$, if every coordinate projection of $X$ has empty interior, then $X$ is contained in a finite union of discrete sets, and so $X$ is constructible. It then follows from \cite[Theorem 3.2]{miller05-tame} that $\Rr$ is constructible. By Theorem~\ref{thm:nip-intro}, its Shelah expansion is also constructible. 
\end{proof}
%To see that C-minimal expansions of C-groups are constructible: I used VC-minimal Swiss cheese decomp.to see constructibility of unary definable sets and then apply arguments I know involving dp-minimality to get the axiom (inf) of Walsberg-Simon. 

\section{Definably complete structures}\label{sec:LO-structures}

In this section we restrict our analysis of uniform structures to the setting of expansions of linearly ordered groups. A linearly ordered structure $(R,<,\ldots)$ is \emph{definably complete} if every definable subset of $R$ has a supremum in $R \cup \{-\infty, +\infty\}$. This definition clearly generalizes the notion of being an expansion of $(\R,<)$. We first show that definable completeness implies definable $\sigma$-compactness, completing the proof of Theorem~\ref{thm:main-intro}. We then recast Theorems~\ref{thm:main-intro} and~\ref{thm:f-intro}, applied to expansions of $(\R,<,+)$, into a descriptive set theoretic dichotomy (Theorem~\ref{thm:Baire-class-one}). Finally, we prove some facts about constructible expansions of definably complete groups, including Theorem~\ref{thm:cons-tame-top}.

Since notions such as constructibility and local o-minimality are trivial in discrete ordered groups, we often focus on densely ordered groups, where we will implicitly rely on the following fact. 

%The next fact is used in proving definable choice in Proposition~\ref{prop:choice} (see $r_1/2$). Because it is trivial, the discrete case is only implicit in the proof of  Lemma~\ref{lem:generic-local-omin} and in Proposition~\ref{prop:local-o-min}. In section~\ref{sec:d-minimal} we completely ignore the discrete case and define d-minimality for dense linear orders. 
\begin{fact}[\cite{MillerIVP}, Proposition 2.2] \label{fct:DC-doag}
Let $\Rr$ be a definably complete expansion of a densely ordered group. Then $\Rr$ is divisible and abelian. 
\end{fact}

Through the next lemma and proposition we show that, in definably complete structures, definable sets are definably compact iff they are closed and bounded, and so in particular definably complete structures are definably $\sigma$-compact. We use an approach analogous to Johnson's~\cite{johnson14}. It is worth noting that one could adapt the proof of Theorem~\ref{thm:main-BP} directly to the definably complete setting by using the analogue of Fact~\ref{fct:comp-basic} (3) for closed and bounded sets proved by Miller~\cite{MillerIVP}. Nevertheless, we consider that developing the basic theory of definable compactness in definably complete structures is of independent interest. Lemma~\ref{lem:intervals-comp} completes the proof of Theorem~\ref{thm:main-intro}.
%For the following, we can develop the theory of definable compactness or instead use Fact 2.2. in~\cite{DMS10}.

\begin{lemma}\label{lem:intervals-comp}
Let $\Rr$ be an expansion of a linear order. If $\Rr$ is definably complete, then any closed and bounded interval $I\subseteq R$ is definably compact. In particular, $\Rr$ is definably $\sigma$-compact.
\end{lemma}
\begin{proof}
Let $\Cc=\{C_b : b\in B\}$ be a downward directed definable family of closed subsets of $I$. By definable completeness, note that 
$
s=\inf_{b\in B} \sup C_b
$
exists. We observe that $s\in \cap \Cc$. Towards a contradiction, let $b_0\in B$ be such that $s\notin C_{b_0}$. By definition of $s$ and definable completeness there exists $s_0=\inf (C_{b_0} \cap (s,\infty))$ with $s_0>s$. By definition of $s$ there exists some $b_1\in B$ such that $\sup C_{b_1}=s_1 < s_0$. By downward directedness let $b_2\in B$ be such that $C_{b_2}\subseteq C_{b_0} \cap C_{b_1}$. Then $\sup C_{b_2} = \max C_{b_2} < s$, which contradicts the definition of $s$. 

Since the family of closed and bounded intervals in $R$ is (uniformly) definable, we get that $\Rr$ is definably $\sigma$-compact.  
\end{proof}

Given a linearly ordered structure $\Rr$, a \emph{box} $D\subseteq R^n$ is a product of $n$ intervals in $R$. We call the box $D$ open, closed or bounded whenever all the intervals have the corresponding property. 

\begin{proposition}[Definable Heine-Borel]\label{prop:Heine-Borel}
Let $\Rr$ be a definably complete structure. A definable set $X$ is definably compact if and only if it is closed and bounded.
\end{proposition}
\begin{proof}
Suppose that $X\subseteq R^n$ is definably compact. By Fact~\ref{fct:comp-basic} (2) it is closed. Let $\{ D_b :b\in B\}$ be the upward directed definable family of open boxes in $R^n$. Then $\{ X\setminus D_b : b\in B\}$ is a downward directed definable family of closed subsets of $X$ with empty intersection. By definable compactness, there exists $b\in B$ such that $X\setminus D_b = \emptyset$, meaning that $X\subseteq D_b$, and so $X$ is bounded. 

Now suppose that $X$ is closed and bounded. Let $D$ be a closed and bounded box such that $X\subseteq D$. By Lemma~\ref{lem:intervals-comp} and Fact~\ref{fct:comp-basic} (4), $D$ is definably compact. By Fact~\ref{fct:comp-basic} (1), the set $X$ is definably compact. 
\end{proof}

%Is being definably complete in fact equivalent to have a being definably $\sigma$-compact and locally definably compact, equivalent to having a definably compact exhaustion? Not sure, but it does seem to be equivalent to having that closed intervals be definably compact. Maybe explore constructible expansions of the p-adics where the balls are definably compact, enough to prove definable choice. Also, for NTP structure being definably sigma-compact means having a definably compact exhaustion (in particular being locally compact). 

A function $f:X\subseteq \R^n\rightarrow \R^m$ is called \emph{first level Borel} if the inverse image of every $F_\sigma$ set by $f$ is $F_\sigma$. These functions belong to the better known collection of Baire class one functions, which are studied in the model-theoretic context in~\cite{comp-set}. Jayne and Rogers~\cite{Jayne-Rogers} showed that, if $X$ is analytic, then $f$ is first level Borel iff it admits a countable partition into continuous functions with locally closed domains. One may show that every function definable in a constructible structure over $\R$ is first level Borel. Theorem~\ref{thm:f-BP} strengthens the partition result of Jayne and Rogers in the case of $\BP$ expansions of $(\R,<,+)$. We now present a version of Theorems~\ref{thm:main-intro} and~\ref{thm:f-intro} that takes the form of a dichotomy given by the definability of a certain ``pathological" first level Borel function. 

\begin{theorem}\label{thm:Baire-class-one}
Let $\Rr$ be an expansion of the real field $(\R,+,\cdot)$, then exactly one holds. 
\begin{enumerate}
    \item The open core $\Ro$ is constructible and every function definable in $\Ro$ is both generically and (finitely) piecewise continuous. 
    \item The open core $\Ro$ defines a first level Borel function that is not (finitely) piecewise continuous.  
\end{enumerate}
\end{theorem}
\begin{proof}
Suppose that statement $(1)$ does not hold. Let $\Nn$ be an $\omega_1$-saturated elementary extension of $\Rr$ and, for every set $A$ definable in $\Rr$, let $A^{\Nn}$ denote its interpretation in $\Nn$. By Remark~\ref{rem:simple-BP} and~\cite[N.B. page 1378]{DMS10}, there exists a definable set $X$, a type-definable set $Z\subseteq X^{\Nn}$, and a definable covering $\{C_b : b\in \R^{>0}\}$ of $X$ by closed and bounded (definably compact by Proposition~\ref{prop:Heine-Borel}) sets, such that $C_b \subseteq C_{b'}$ for every $b<b'$ and moreover $C^{\Nn}_b \cap Z$ has empty interior in $Z$ for every $b\in N^{>0}$. 

Since (1) does not hold, then $\Ro$ is not o-minimal and so, by~\cite[Proposition 2.14]{DMS10}, $\Ro$ defines an infinite discrete subset $D$ of $\R$. By \cite[Theorem B]{hie-discrete}, we may assume that $D$ is closed. Using the group operation we may further assume that $D$ is an unbounded subset of $\R^{>0}$.
Let us consider the definable function $f:X\rightarrow \R$ given by 
\[
f(x)=\min\{ b\in D : x \in C_b\}.
\]
This map is clearly first level Borel, since it is a countable union of constant maps on locally closed sets of the form $C_{b}\setminus C_{b'}$, where $b$ is the successor of $b'$ in $D$. Because $D$ is closed and discrete and the fibers $f^{-1}(b)$, for $b\in D$, are locally closed, it is easy to see that the graph of $f$ is constructible, hence $f$ is definable in $\Ro$. We show that it is not piecewise continuous. 

Let $X_0=X$ and, for every $n$, let $X_{n+1}$ be the points $x\in X$ satisfying that, for every $r\in \R^{>0}$, there exists some $y\in X_n$ with $\|x-y\|_\infty<r$ and $f(y)>f(x)$. Clearly each set $X_n$ is definable. 

\begin{claim}\label{claim:Xn}
For every $n$ it holds that $X_n\neq \emptyset$. 
\end{claim}
\begin{claimproof}
We show by induction on $n$ that $Z\subseteq X_n^{\Nn}$ for every $n$. The base case is trivial. Let $x\in Z$ and let $W$ be an open neighborhood of $x$ in $Z$. Since $C_{f(x)}^{\Nn}$ is nowhere dense in $Z$, there exists some $y\in W \setminus C_{f(x)}^{\Nn}$, which implies that $f^{\Nn}(y)>f^{\Nn}(x)$. By induction hypothesis we also have that $y\in X_n^{\Nn}$, and so it follows that $x\in X_{n+1}^{\Nn}$.
\end{claimproof}

Now, towards a contradiction, let $\Yy$ be a finite partition of $X$ into subsets where the restriction of $f$ is continuous. Fix $n=|\Yy|$ and, by Claim~\ref{claim:Xn}, some $x_0\in X_{n}$. Let $Y_0\in \Yy$ be such that $x_0\in Y_0$. By definition of $f$ and continuity on $Y_0$, note that $f|_{Y_0}$ is locally constant. That is, there exists a neighborhood $W_0$ of $x_0$ in $X$ such that $f(Y_0 \cap W_0)=f(x_0)$. On the other hand, since $x_0\in X_{n}$, there exists some $x_1\in W_0 \cap X_{n-1}$ with $f(x_1)>f(x_0)$. In particular we have that $x_1\notin Y_0$. 

Continuing recursively we may find, for every $i\leq n$, some $x_i\in X$, a set $Y_i\in \Yy$, and an open set $W_i$ in $X$ satisfying that 
\begin{enumerate}
\item $x_i \in X_{n-i} \cap Y_i \cap \bigcap_{j\leq i} W_j$, 
\item $f(Y_i \cap W_i)=f(x_i)$, 
\item $f(x_j)<f(x_i)$ for every $j<i$. 
\end{enumerate}
In particular note that, for every $i\leq n$, we have that $x_i \in Y_i \setminus \bigcup_{j < i} Y_j$, and so $Y_j \neq Y_i$ for every $0\leq j < i\leq n$. This however contradicts that $|\Yy|=n$.
\end{proof} 
%Towards a contradiction, let $\Yy$ be a finite partition of $X$ into sets where the restriction of $f$ is continuous. By expanding the language if necessary we may assume, Without loss of generality, that the sets in $\Yy$ are definable. Let $\Yy^{\Nn}$ be the family of interpretations in $\Nn$ of the sets in $\Yy$. 

%Let $x_0\in Z$, and let $Y_0\in \Yy^{\Nn}$ be such that $x_0\in Y_0$. By definition of $f^{\Nn}$ and continuity on $Y_0$, there exists a neighborhood $U_0$ of $x_0$ in $X^{\Nn}$ such that $f^{\Nn}(Y_0 \cap U_0)=f^{\Nn}(x_0)$, in particular $Y_0 \cap U_0\subseteq C^{\Nn}_{f(x_0)}$. However, since $C^{\Nn}_{f(x_0)}$ has empty interior in $Z$, there is some $y_1 \in  Z \cap U_0$ with $y_1\notin C^{\Nn}_{f(x_0)}$, meaning $f^{\Nn}(y_1)>f^{\Nn}(y_0)$ and $y_1\notin Y_0$. We may now proceed recursively, finding sequences $(x_n)_n$ and $(Y_n)_n$ of distinct points and sets in $Z$ and $\Yy^{\Nn}$ respectively, and a decreasing sequence $(U_n)_n$ of open subsets of $X^{\Nn}$, such that $x_n\in Y_n \cap Z \cap U_n$ and $f^{\Nn}(Y_n \cap U_n)=f^{\Nn}(x_n)$ for every $n$. This, however, contradicts the fact that $\Yy$ is finite. 

We devote the rest of the section to proving facts about constructible definably complete expansions of ordered groups.  
A linearly ordered structure $(R,<,\ldots)$ is \emph{generically locally o-minimal} if, for every definable set $X\subseteq R$, there exists a definable subset $Y\subseteq X$ that is open and dense in $X$ and such that, for every $x\in Y$, there is an interval $I$ containing $x$ satisfying that $I\cap X$ is either $\{x\}$ or an interval. We show that every constructible definably complete structure is generically locally o-minimal. This is essentially a known consequence of the fact that constructibility implies the non-definability of a perfect nowhere dense set (see e.g. Lemma 3.26 in~\cite{for21}). 

\begin{lemma}\label{lem:generic-local-omin}
Let $\Rr$ be a definably complete constructible structure. Then $\Rr$ is generically locally o-minimal.      
\end{lemma}
\begin{proof}
Let $X\subseteq R$ be a definable set. Let $Y$ be the union of the interior and the isolated points in $X$. It suffices to show that $Y$ is dense in $X$. 

Suppose that $Y$ is not dense, then, by passing to $X\setminus cl(Y)$ if necessary, we may assume that $X$ has no interior and no isolated points. Since $\Rr$ is constructible, $cl(X)$ must also have empty interior (Lemma~\ref{lem:cons-noise0}), and so we may assume that $X$ is closed. 

Let $C_0$ be the set of points $y$ in $X$ such that $(y',y)\cap X=\emptyset$ for some $y'\in R$, and $C_1$ be the set of points $z$ in $X$ such that $(z,z')\cap X=\emptyset$ for some $z'\in R$. Clearly both $C_0$ and $C_1$ are definable and disjoint (because $X$ has no isolated points). Furthermore, since $X$ has empty interior and $\Rr$ is definably complete, they are both dense in $X$. Hence, $C_0$ is both dense and codense in $X$, contradicting that $\Rr$ is constructible. 
\end{proof}

The next proposition is a direct consequence of Lemma~\ref{lem:generic-local-omin} and Theorem~\ref{thm:main-intro}. It was established for the open core of $\NIP$ expansions of $(\R,<,+)$ in~\cite[Theorem C]{Wal22}.   

\begin{proposition}\label{prop:generic-local-o-min}
Let $\Rr$ be an $\NTP$ definably complete expansion of an ordered group. Then $\Ro$ is generically locally o-minimal.  
\end{proposition}

Definable choice in constructible definably complete expansions of fields was proved by Fornasiero~\cite[Section 4]{for21}. Miller~\cite{miller06-choice} proved it in d-minimal expansions of groups, and Hieronymi-Nell-Walsberg~\cite{HieNelWal18} in noiseless expansions of the real ordered group that do not define a Cantor set (equivalently, by Theorem 8.1 in~\cite{Wal22}, in strongly noiseless expansions of $(\R,<,+)$). We prove it in strongly noiseless definably complete expansions of groups. 
%In"NOTES ON DEFINABLY COMPLETE LOCALLY O-MINIMAL EXPANSIONS OF ORDERED GROUPS" Fujita claims to prove definable choice in definably complete expansion of an ordered groups, but this is known to now be not true in this generality, and in the proof he seems to be assuming local o-minimality.  

%It would be interesting to know whether or not constructible expansions of the p-adics have DSF. Definable Skolem functions was proved first by van den Dries in ``Algebraic Theories with Definable Skolem Functions" (no explicit formula). See also ``A Note on Definable Skolem Functions" from Scowcroft.

%Does definable choice in constructible structure imply that there are universally axiomatizable? Our definable choice functions are 0-definable? I think yes. But then we get universal axiomatization?? (see Lemma 3.1 in Definable choice for a class of weakly o-minimal theories.) Indeed, every 0-definable set is defined by a single formula, but we need to make them function symbols explicitly. So choice functions (for 0-definable sets) are given by function symbols.

\begin{proposition}[Definable choice]\label{prop:choice}
Let $\Rr$ be a definably complete expansion of an ordered group with a positive constant which we denote $1$. Let $A\subseteq R$, and let $X\subseteq R^{m+n}$ be an $A$-definable set such that, for every $b\in R^m$, the fiber $X_b$ is constructible. Then there exists an $A$-definable map $f:R^m\rightarrow R^n$ satisfying that, for every $b\in \R^m$, if $X_b\neq \emptyset$, then $f(b)\in X_b$, and moreover, for every pair $b,b'\in R^m$, if $X_b = X_{b'}\neq \emptyset$, then $f(b)=f(b')$.  

As a consequence, if $\Rr$ is strongly noiseless, then it has definable choice functions. 
\end{proposition}
\begin{proof}
We first explain how the second paragraph follows from the first one. Suppose that $\Rr$ is strongly noiseless and let $\{ Y_b : b\in R^m\}$ be a definable family of nonempty sets. For each $b\in R^m$, let $Y'_b = cl(Y_b)\setminus cl(\partial Y_b) = Y_b \setminus \partn{2}Y_b$. By strong noiselessness, for every $b\in R^m$ the constructible set $Y'_b$ is nonempty. The first paragraph of the proposition yields the existence of a definable choice function for $\{ Y'_b : b\in B\}$, which is also a choice function for $\{ Y_b : b\in R^m\}$. 

We now prove the first paragraph. Henceforth, for any $x\in R^n$ and $r\in R^{>0}$, let $B(x,r)$ denote the open ball with center $x$ in the supremum $R$-valued metric, and let $\overline{B(x,r)}$ be the corresponding closed ball. For any set $S\subseteq R^n$, let $B(S,r)=\bigcup_{x\in S} B(x,r)$. Fix $b\in B$ with $X_b\neq \emptyset$. Suppose that $X_b$ is closed. Let $r_0=\inf \{ s\in R^{>0} : B(0,s) \cap X_b \neq \emptyset\}$. We may choose $f(b)$ to be the lexicographic minimum of the closed and bounded set $X_b \cap \overline{B(0,r_0 + 1)}$. Observe that this covers the case where $\Rr$ is discrete, so onward we may assume that it is dense. Now suppose that $X_b$ is not closed. We prove the proposition by showing that there exists a nonempty definable closed subset of $X_b$, which is moreover $Ab$-definable uniformly in $b$. 

Let $Z=cl(\partial X_b)$. Since $X_b$ is not closed, this set is nonempty. Since $X_b$ is constructible, we have that $W=X_b \setminus Z \neq \emptyset$. Let $r_1=\sup\{ s \in R^{>0} : W\setminus B(Z,s) \neq\emptyset\}$ if the supremum is finite, and otherwise let $r_1=1$. Now set $C_b:=X_b \setminus B(Z,r_1/2)$. Note that $C_b = cl(X_b)\setminus B(Z,r_1/2)$, and so $C_b$ is closed. Finally, observe that $C_b$ is $Ab$-definable uniformly in $b$. 
\end{proof}

Proposition~\ref{prop:choice} does not generalize to non definably complete constructible structures. The structure $(\Q, +, <, 1, I)$, where $I=(-\sqrt{2}, \sqrt{2})\cap \Q$, has weakly o-minimal theory, and thus it is dp-minimal (see e.g. \cite{simon:guide-NIP}). Applying either Corollary~\ref{cor:shelah-exp} or the literature in viscerality developed in \cite{sim-wals19}, we get that $(\Q, +, <, I)$ is constructible. However, by~\cite{EHK-choice}, this structure does not have definable Skolem functions. 
% $(\Q, +, <, I)$ is non-valuational.
% For definitions of definable choice vs definable Skolem functions see: DEFINABLE CHOICE FOR A CLASS OF WEAKLY O-MINIMAL THEORIES MICHAEL C. LASKOWSKI \& CHRISTOPHER S. SHAW.
The proposition also fails to generalize to noisy definably complete structures. The structure $(\R,<,+,1,\Q)$ is an example of an $\NIP$ structure with o-minimal open core that fails to have definable Skolem functions \cite[5.4.]{DMS10}. 
%The above sentence is explicitly pointed out in \cite{HieNelWal18}

In his proof of Proposition~\ref{prop:choice} for definably complete expansions of ordered fields, Fornasiero~\cite{for21} shows and uses that every definable constructible set is the projection of a definable closed set. A similar fact was also proved in expansions of the real field by Miller and Speissegger \cite[page 202]{MS99}. We point out, through the following example, that this fact is no longer true if one drops the field assumption.  

\begin{example}
 Consider the real ordered group $\Rr=(\R,<,+)$. It is a well-known fact that this structure has quantifier elimination and thus it is o-minimal and every definable function $f:\R\rightarrow \R$ is piecewise linear. We show that the interval $I=(0,\infty)$ is not the projection of a closed definable set. Towards a contradiction let $C\subseteq \R^n$ be a closed definable set and $\pi:C\rightarrow \R$ be a projection, chosen without loss of generality to the first coordinate, such that $\pi(C)=I$. For each $t\in I$, let $C_t$ be the corresponding nonempty fiber of $C$. Since $C$ is closed and $\pi(C)=I$, observe that $\gamma(t)$ cannot converge in $\R^n$ as $t\rightarrow 0$. By o-minimality, there must exist a coordinate function $\gamma_i:I\rightarrow \R$ of $\gamma$ that converges to either $-\infty$ or to $\infty$ as $t\rightarrow 0$. This however contradicts the linearity of $\Rr$.    
\end{example}

%Using Proposition~\ref{prop:choice} maybe you get that in structures where every definable set is the projection of a constructible set (i.e. interdefinable with open core while open core is existentially closed) there are definable Skolem functions. 

Given a nonempty set $X\subseteq \R^n$, the \emph{naive dimension} of $X$, which we denote henceforth by $\dim(X)$, is the maximum between zero and the largest $0< d \leq n$ such that there exists some projection $\pi:\R^n \rightarrow \R^d$ satisfying that $\pi(X)$ has nonempty interior. By convention $\dim \emptyset = -\infty$.  

The dimension theory for $D_\Sigma$ sets (see Remark~\ref{rem:simple-BP}) was developed in~\cite{comp-set}, and further analyzed in generically locally o-minimal expansions of $(\R,<,+)$ in the preprint version~\cite{wal22-preprint} of~\cite{Wal22}. Through the next proposition we showcase that naive dimension is well-behaved in constructible expansions of $(\R,<,+)$. In particular, by Theorem~\ref{thm:main-intro}, this applies to the open core of every $\NTP$ expansion of $(\R,<,+)$. By~\cite[Section 10]{wal22-preprint}, the proposition holds in the more general setting of strongly noiseless expansions of $(\R,<,+)$. We nevertheless state it in the constructible setting because its proof follows mostly elementarily from~\cite{comp-set}.
%Things that don't hold in our setting: No small boundaries. No "dim A= dim B, for A subset of B, implies that A has nonempty interior in B". Fails even for zero-dimensional sets.   

%Notion of $D_\Sigma$ is used in the next proposition.
\begin{proposition}\label{prop:dim}
Let $\Rr$ be a constructible expansion of $(\R,<,+)$. Let $X, Y\subseteq \R^n$ be definable sets. The following hold. 
\begin{enumerate}
    \item \label{itm:dim-Lebesgue} The naive dimension $\dim(X)$ of 
    $X$ is equal to its classical topological dimension\footnote{That is, small inductive dimension, large inductive dimension, or Lebesgue covering dimension, all of which are identical for subsets of $\R^n$.}.
    \item \label{itm:dim-union} $\dim(X\cup Y)=\max\{ \dim X, \dim Y\}$.
    \item \label{itm:dim-cl} $\dim X = \dim cl(X)$.
    \item \label{itm:dim-f} If $f:X\rightarrow \R^m$ is a definable function, then $\dim f(X) \leq \dim X$. In particular dimension is maintained under definable bijections.
    \item \label{itm:dim-fiber} Let $0<m<n$, and let $Z\subseteq \R^m$ denote the projection of $X$ to some $m$ coordinates. For every $a\in Z$, let $X_a$ denote the corresponding fiber of $X$. For every $0\leq d \leq n-m$, let $S(d)=\{ a\in Z : \dim X_a =d\}$. Then each set $S(d)$ is definable and 
    \[
    \dim X = \max_{0\leq d \leq n-m} \dim S(d) + d.  
    \]
\end{enumerate}
\end{proposition}
\begin{proof}
By~\cite[1.10]{DMS10} every constructible definable set is $D_\Sigma$ (see Remark~\ref{rem:simple-BP}). Furthermore, by Lemma~\ref{lem:cons-noise0}, $\Rr$ cannot define a dense $\omega$-order in the sense of~\cite{comp-set}. It follows that statement~\eqref{itm:dim-Lebesgue} is given by~\cite[Proposition 5.7]{comp-set}, and statement~\eqref{itm:dim-f} follows from~\cite[Theorem E]{comp-set}, where the results are achieved for $D_\Sigma$ sets in expansions of $(\R,<,+)$ that do not define a dense $\omega$-order. 

While statements~\eqref{itm:dim-union} and~\eqref{itm:dim-cl} also follow from~\cite{comp-set}, we include the proofs, which are straightforward in our context. For~\eqref{itm:dim-union}, note that clearly $\dim(X\cup Y) \geq \max\{ \dim X, \dim Y\}$. For the other inequality, let $\pi$ be a projection such that $\pi(X\cup Y)=\pi(X)\cup\pi(Y)$ has nonempty interior $U$. By constructibility, either $\pi(X)$ or $\pi(Y)$ must have nonempty interior, since otherwise both sets would be dense and codense in $U$, so $\dim(X\cup Y) \leq \max\{ \dim X, \dim Y\}$. For~\eqref{itm:dim-cl}, we again prove the non-trivial implication, namely $\dim cl(X) \leq \dim X$. Let $\pi$ be a projection such that $\pi(cl(X))$ has nonempty interior. Then, by constructibility and the fact that $\pi(X)$ is dense in $\pi(cl(X))$, we get that $\pi(X)$ has nonempty interior too. 

For statement \eqref{itm:dim-fiber} note that, by \eqref{itm:dim-f}, we may substitute $X$ by any coordinate permutation, and so we may assume that $Z$ is the projection to the first $m$ coordinates. Under this assumption, statement \eqref{itm:dim-fiber} was proved for generically locally o-minimal expansions of $(\R,<,+)$ in~\cite[Theorem 10.16]{wal22-preprint}, and so it follows by Lemma~\ref{lem:generic-local-omin}.
\end{proof}

\section{Strong uniform structures} \label{sec:strong}

In this section we strengthen the $\NTP$ condition to the property of being strong. Among $\NIP$ structures, being strong is equivalent to being strongly dependent. Following the treatment in~\cite{Wal22}, we characterize all strong expansions of $(\R,<,+)$. More generally, we show that every strong expansion of a definably complete group has locally o-minimal open core (Theorem~\ref{thm:strong-intro}). In the wider context of strong uniform structures, we begin the section by showing that constructibility is equivalent to strong noiselessness, and that having constructible open core is maintained under elementary equivalence. 

%\begin{proposition}
%Let $(X,\tau)$ be a definable topological space and $Y$ be a definable subset with %$\taurk(Y)<\omega$. Exactly one of the following holds:
%\begin{enumerate}
%    \item $Y$ is constructible. 
%    \item There exists some $n$ such that $\partial^{(m)}Y = \partial^{(m+2)}Y \neq \emptyset$ for every $m\geq n$. In particular, $\partial^{(n)} X$ is dense a codense in $cl(\partial^{(n)} X)$. 
%\end{enumerate}
%\end{proposition}
%\begin{proof}  %the proof is not exactly the statement
%We assume that $Y$ is not constructible. Since $\taurk(Y)
%<\omega$ and the family of sets $\partial^{(2i)}Y$ is a decreasing chain, there must exists %some even number $n$ such that $\taurk(\partial^{(n)} Y)= \taurk(\partial^{(n+2)} Y)$. %Recall that the $\tau$-rank of a set is always the same as the $\tau$-rank of its closure, %so $\taurk(cl(\partial^{(n)} Y))= \taurk(cl(\partial^{(n+1)} Y))$. By Lemma~\ref{lem:C-rk-%interior}, $cl(\partial^{(n+1)} Y)$ has nonempty interior in $cl(\partial^{(n)} Y)$. Let %$U$ be a definable open set such that $Z=U\cap cl(\partial^{(n)} Y) \subseteq %cl(\partial^{(n+1)} Y)$. We show that $Y\cap Z$ is dense a codense in $Z$. 

%Since $Z$ is open in $cl(\partial^{(n)} Y)$ any $x\in Z\cap \partial^{(n+1)} Y$ is clearly %in the closure of $Z \cap \partial^{(n)} Y= Z \cap \partial^{(n+1)} Y$. The converse is %analogous.   
%\end{proof}

In the proof of the next lemma, we recover the notation $(U\setminus V)[Y]=U[Y]\setminus V[Y]$ used in Section~\ref{sec:ntp}.

%Localized inp-pattern appears here:
\begin{lemma}\label{lem:closed-chain}
Let $\Mm$ be a uniform structure. Let $\kappa>1$ be a cardinal and $\{ \X{\alpha} : \alpha < \kappa\}$ be a family of nonempty definable subsets of $M^d$ satisfying that, for all $\alpha < \beta < \kappa$,
\begin{enumerate}
    \item $\X{\beta} \subseteq \X{\alpha}$,
    \item $\X{\beta}$ is nowhere dense in $\X{\alpha}$ (in the subspace topology).
\end{enumerate}
Then $\X{0}$ admits an inp-pattern of depth $\kappa-1$. 
\end{lemma}
\begin{proof}
Let $\BU$ be a definable basis for the product uniformity on $M^d$. Let $\{ \X{\alpha} : \alpha < \kappa\}$ be as in the lemma. By passing to the closures if necessary, we may assume that every set $\X{\alpha}$, for $0<\alpha<\kappa$, is closed in $\X{0}$. For each $0<\alpha<\kappa$, let $\varphi_\alpha$ be a (partitioned) formula that defines the family of sets of the form $(U\setminus V)[\X{\alpha}]$, for $U,V\in \BU$. We prove the existence of an inp-pattern of depth $\kappa-1$ witnessed by the formulas $\varphi_\alpha$. Specifically, we describe, for every $m>1$ and $0<\alpha_1<\cdots\alpha_m < \kappa$, entourages $U_{i,j}$, for $1\leq i\leq m$ and $1\leq j\leq m+1$, such that the family of sets $\{ (U_{i,j} \setminus U_{i,j+1})[\X{\alpha_i}] : i, j\leq m\}$ forms an inp-pattern of depth $m$ and length $m$ in $\X{0}$ (namely conditions (1) and (2) in the definition of inp-pattern in Section~\ref{sec:prelim} are satisfied, and moreover the families described in condition (2) are consistent with $\X{0}$). For simplicity of notation, and since there will be no room for confusion, we set $\alpha_i=i$ for every $i\leq m$.

Let us fix $m>1$. For any $n\leq m$, suppose that we have defined entourages $U_{i,j}$ for $n\leq i \leq m$ and $j\leq m+1$, and consider the following two statements.
\begin{description}
    \item[$(i)_n$] Each family $\{ (U_{i,j} \setminus U_{i,j+1})[\X{i}] : j\leq m\}$, for $n\leq i\leq m$, is pairwise disjoint. 
    \item[$(ii)_n$] For every map $\eta:\{n,\ldots, m\}\rightarrow \{1,\ldots, m\}$, the intersection of the family $\{ (U_{i,\eta(i)} \setminus U_{i,\eta(i)+1})[\X{i}] : n\leq i\leq m \}$ has nonempty interior which intersects $\X{n-1}$.
\end{description}
We prove the lemma by proving $(i)_1$ and $(ii)_1$. We proceed by reverse induction on $n\leq m$. We will use a simple claim. 

\begin{claim}\label{claim:UV}
Let $\Ff$ be a finite family of nonempty open sets in $M^d$, each intersecting $\X{n}$. For any entourage $U\in \BU$, there exists another entourage $V\in \BU$ such that, for all $W\in \Ff$, the \textbf{interior} of the set $W \cap (U\setminus V)[\X{n}]$ intersects the set $\X{n-1}$.
\end{claim}
\begin{claimproof}
Since $\X{n}$ has empty interior in $\X{n-1}$, for each $W\in \Ff$ there exists a point $x_W \in W\cap U[\X{n}] \cap  \X{n-1} \setminus \X{n}$. Since $\X{n}$ is closed, there exists $V \in \BU$ such that $\{x_W : W\in \Ff\}$ is disjoint from the closure of $V[\X{n}]$. The claim follows. 
\end{claimproof}

Let $U_{m,1}\in \BU$ be any entourage. By repeated application of Claim~\ref{claim:UV} (with $\Ff=\{M^d\}$) there exists a decreasing sequence of entourages $U_{m,j}$, for $j\leq m+1$, satisfying that the interior of each set $(U_{m,j} \setminus U_{m,j+1})[\X{m}]$ intersects the set $\X{m-1}$. In particular condition $(ii)_m$ is clearly satisfied. Since the sequence of entourages is decreasing, condition $(i)_m$ is also immediate. 

We now fix $n<m$, and by induction hypothesis we assume that both $(i)_{n+1}$ and $(ii)_{n+1}$ hold, witnessed by entourages $U_{i,j}$, for $n<i\leq m$ and $j\leq m+1$. For each map $\eta:\{n+1,\ldots, m\}\rightarrow \{1,\ldots, m\}$, let us denote the interior of the intersection $\bigcap \{(U_{i,\eta(i)} \setminus U_{i,\eta(i)+1})[\X{i}] : n< i\leq m \}$ by $W_\eta$. By condition $(ii)_{n+1}$, every such open set $W_\eta$ intersects $\X{n}$. Let $U_{n,1}\in \BU$ be any entourage. By repeated application of Claim~\ref{claim:UV}, we may find a decreasing sequence of entourages $U_{n,j}$, for $j\leq m+1$, satisfying that, for any $\eta:\{n+1,\ldots, m\}\rightarrow \{1,\ldots, m\}$ and $j\leq m$, the interior of the set $W_\eta \cap (U_{n,j} \setminus U_{n,j+1})[\X{n}]$ intersects the set $\X{n-1}$. Condition $(ii)_n$ clearly follows. Furthermore, since once again the entourages $U_{n,j}$ are decreasing, condition $(i)_n$ is immediate. 
\end{proof}

Lemma~\ref{lem:closed-chain} yields that strong uniform structures have a property known as the ``closed chain condition". For an exploration of the topological consequences of this property, the interested reader is directed to~\cite{wal22-preprint}. 

We now present two short lemmas that will be used in proving Theorem~\ref{thm:strong-noise}.

\begin{lemma}\label{lem:fr-interior}
Let $X$ be a subset of a topological space. If $cl(\partial X)$ has nonempty interior $Z$ in $cl(X)$, then $X$ is dense and codense in $Z$. In particular $X$ is not constructible.  
\end{lemma}
\begin{proof}
Let $W$ be a nonempty open set such that $W\cap cl(X) = Z \subseteq cl(\partial X)$. Note that $cl(X)= X \cup \partial X$ and $cl(\partial X)=\partial X \cup \partn{2}X$, where the unions are disjoint and $X\cap cl(\partial X)=\partn{2}X$. Hence $W\cap X \subseteq \partn{2}X$. So $W\cap X \subseteq cl(W\cap \partial X)$. Since $W\cap \partial X \subseteq cl(W\cap X)$, it follows that $X$ is dense and codense in $W\cap cl(X) = Z$. The last sentence in the lemma follows from Lemma~\ref{lem:cons-noise0}.
\end{proof}

\begin{lemma}\label{lem:strong-noise}
Let $\Mm$ be a uniform structure and $X$ be a definable set that is strong (i.e. $X$ does not admit an inp-pattern of depth $\omega$). Exactly one of the following holds.
\begin{enumerate}
    \item $X$ is constructible. 
    \item There exists a definable set $Z\subseteq cl(X)$ such that $X$ is dense and codense in $Z$. 
\end{enumerate}
\end{lemma}
\begin{proof}
Statements $(1)$ and $(2)$ are mutually exclusive by Lemma~\ref{lem:cons-noise0}. Let $X$ be a non-constructible definable set. Consider the decreasing family of nonempty (Fact~\ref{fct:constructible0}) definable closed sets of the form $cl(\partn{n}X)$. Using that $X$ is strong and applying Lemma~\ref{lem:closed-chain}, we get that there exists some $n$ such that $cl(\partn{n+1}X)$ has nonempty interior $Z$ in $cl(\partn{n}X)$. By Lemma~\ref{lem:fr-interior}, the set $\partn{n}X$ is then dense and codense in $Z$. Since the intersection $X\cap cl(\partn{n}X)$ is either $\partn{n}X$ or $\partn{n+1}X$,  we conclude that $X$ is dense and codense in $Z$.
\end{proof}

%By the above lemma it also holds that, in uniform structures, constructibility is definable in families. 

%Question: can the above lemma be improved by changing (2) to the property that $\partn{2n}Y=\partn{2n+2}Y \neq \emptyset$ for some $n$? 

\begin{theorem}\label{thm:strong-noise}
Let $\Mm$ be a strong uniform structure. Then $\Mm$ is strongly noiseless if and only if it is constructible. If $\Mm$ is an expansion of $(\R,<,+)$, then these are equivalent to $\Mm$ being noiseless.  
\end{theorem}
\begin{proof}
The equivalence between strong noiselessness and constructibility follows directly from Lemma~\ref{lem:strong-noise}. For the last sentence, Walsberg proves in~\cite[Theorem 8.1]{Wal22} that an expansion $\Mm$ of the real ordered group is strongly noiseless if and only if it is noiseless and it does not definable a Cantor set, however, by~\cite[Theorem B]{HW18}, if $\Mm$ defines a Cantor set then it defines a copy of the two sorted structure $(\mathcal{P}(\N), \N,\in, +1)$, and in particular $\Mm$ cannot be strong. 
\end{proof}
 
In general, if a topological structure does not have constructible open core, then there exists a definable constructible set $X$ and some projection $\pi(X)$ such that $\partn{n}(\pi(X)) \neq \emptyset$ for every $n$. From this it can be derived that a theory has constructible open core if and only if any (every) $\omega_1$-saturated model has constructible open core. Notwithstanding Remark~\ref{rem:T-open-core}, it is open whether having constructible open core is closed under elementary equivalence within $\NTP$ uniform structures. We provide a positive answer for strong structures. 

\begin{proposition}\label{prop:cons-OC}
Let $\Mm$ be a strong uniform structure. If $\Nn\equiv \Mm$ and $\Mm$ has constructible open core, then $\Nn$ has constructible open core. 
\end{proposition}
\begin{proof}
Suppose that $\Nn$ does not have constructible open core. Then there exists a constructible definable set $X$ and a projection $Y=\pi(X)$ of $X$ that is not constructible. By Lemma~\ref{lem:strong-noise}, there exists a definable set $Z$ such that $Y$ is dense and codense in $Z$. Let $n$ be such that $\partn{n}X =\emptyset$ (see Fact~\ref{fct:constructible0}). Let $\varphi(x,y)$ and $\psi(z,w)$ be formulas such that $X=\varphi(N,b)$ and $Z=\psi(N,c)$, for some parameters $b$ and $c$. The theory of $\Nn$ contains the formula stating that there exists some $y$ and $w$ such that $\varphi(x,y)$ defines a set $X'$ with $\partn{n}X'=\emptyset$ and $\psi(z,w)$ defines a set $Z'$ such that $Y'=\pi(X')$ is dense and codense in $Z'$. It follows that every model of $Th(\Nn)$ fails to have constructible open core.   
\end{proof}

\begin{question}
Let $\Mm$ be an $\NTP$ uniform structure whose open core $\Mo$ is constructible. If $\Nn\equiv \Mm$, is $\No$ constructible? 
\end{question}

We now begin the proof of Theorem~\ref{thm:strong-intro}, which we divide into Proposition~\ref{prop:local-o-min} and Corollary~\ref{cor:wal}. We will need some facts from~\cite{DG17, Wal22}. 

\begin{fact}[\cite{DG17}, Corollary 2.13] \label{fct:DG-acc-point}
Let $\Mm$ be a monster model of a strong definably complete theory expanding densely ordered abelian groups. Then:
\begin{enumerate}
    \item Every unary definable discrete set has no accumulation points. 
    \item Every unary definable nowhere dense set is discrete. 
\end{enumerate}
\end{fact}

A linearly ordered structure $(M,<,\ldots)$ is \emph{locally o-minimal} if, for every definable set $X\subseteq M$ and every $a\in X$, there exists an interval $I$ containing $a$ such that $I\cap X$ is either $\{x\}$ or an interval. 

%In the next proposition discrete case is implicit and trivial. 
\begin{proposition}\label{prop:local-o-min}
Let $\Mm$ be a strong definably complete expansion of an ordered group. Then the open core $\Mo$ is locally o-minimal.
\end{proposition}
\begin{proof} 
Let $X\subseteq M$ be a set definable in $\Mo$ and $a\in X$ be a point such that, for every interval $I$ containing $a$, the intersection $I\cap X$ is not finite nor an interval. By Theorem~\ref{thm:main-intro}, $\Mo$ is constructible, and in particular noiseless. It follows that $Y=cl(X)\setminus int(X)$ has empty interior. We may assume that $\Mm$ is a monster model. Hence, by Fact~\ref{fct:DG-acc-point}, $Y$ is discrete. We show that $a$ is an accumulation point of $Y$, from where it follows, by Fact~\ref{fct:DG-acc-point}, that $\Mm$ is not strong.

Let us fix an interval $I\ni a$. Observe that either $I \cap X \cap (a,+\infty)$ or $I \cap X \cap (-\infty, a)$ must not be finite nor an interval. We assume the former, being the latter case analogous. Hence there exist $x, y\in X\cap I$, with $a<x<y$, such that the interval $[x,y]$ is not a subset of $X$. By definable completeness, let $z =\inf \, \left( [x,y]\setminus X \right)$. Then $z \in Y \cap I$, completing the proof.         
\end{proof}

The following fact can be extracted from~\cite{Wal22} (see the proof of Theorem 9.15), where it is shown for strongly dependent structures but using results from~\cite{DG17} for strong structures. 

\begin{fact}[\cite{Wal22}]\label{fct:strong-expansion}
Let $\Mm$ be an expansion of $(\R,<,+)$. The following are equivalent. 
\begin{enumerate} 
    \item $\Mm$ is strong and noiseless.
    \item $\Mm$ is either o-minimal or locally o-minimal and interdefinable with $(\R,<,+,\Bb, \alpha\Z)$, for some $\alpha>0$, where $\Bb$ is a collection of bounded sets satisfying that $(\R,<,+,\Bb)$ is o-minimal.
\end{enumerate}
\end{fact}

The next corollary was proved independently by Hieronymi and Walsberg in unpublished work. 

\begin{corollary}\label{cor:wal}
Let $\Mm$ be an expansion of $(\R,<,+)$. The following are equivalent. 
\begin{enumerate}
    \item $\Mm$ is a strong expansion by constructible sets. 
    \item $\Mm$ is constructible and has dp-rank at most $2$.
    \item $\Mm$ is either o-minimal or locally o-minimal and interdefinable with $(\R,<,+,\Bb, \alpha\Z)$, for some $\alpha>0$, where $\Bb$ is a collection of bounded sets satisfying that $(\R,<,+,\Bb)$ is o-minimal.
\end{enumerate}
\end{corollary}
\begin{proof}
We first observe that $(3)\Rightarrow(2)$. It is well known that o-minimal structures have dp-rank $1$ and are constructible. In the proof of \cite[Corollary 11.10]{wal22-preprint}, it is shown that the locally o-minimal structure $(\R,<,+,\Bb, \alpha\Z)$ described in $(3)$ has dp-rank $2$. By either Theorem~\ref{thm:main-intro} or the literature on local o-minimality, it is also constructible. 

$(2)\Rightarrow(1)$ is obvious. For $(1)\Rightarrow (3)$, if $\Mm$ is a strong expansion by constructible sets then, by Theorem~\ref{thm:main-intro}, it is constructible, and in particular noiseless. Hence $(3)$ is given by Fact~\ref{fct:strong-expansion}. 
\end{proof}

%QUESTION: In strong definably complete expansions $\Mm$ of ordered groups $(M,<,+)$ (they are locally o-minimal), could it be the case that the $\Mm$ is always $(M,<,+,S)$-minimal, where $S$ is ANY infinite discrete definable set? 

The structures $(\R,+,<,\Q)$ and $(\R,+,<,\Z)$ both have dp-rank $2$ (see Section 3.1 in \cite{DG17}). If we drop the constructibility assumption in Corollary~\ref{cor:wal} and impose a bound of $2$ on the dp-rank, then these structures exemplify the only two possibilities described by the next proposition.  

% For the above results see Lemma 3.5 and the comment at the end of Section 3.1 in \cite{DG17}.

%They also have o-minimal and $d$-minimal open core respectively

%In~\cite[Proposition 3.14]{DG17} it is shown that $(R,+,<,\R)$, where $R$ is the universe of an elementary extension of the real ordered group, has dp-rank $2$. 

\begin{proposition}\label{prop:dprk2}
Let $\Rr$ be an expansion of $(\R,<,+)$ of dp-rank at most $2$. Exactly one of the following holds. 
\begin{enumerate}
    \item The open core $\Ro$ is o-minimal. 
    \item $\Rr$ is interdefinable with $(\R,+,<,\Bb, \alpha\Z)$, for some $\alpha\in \R^{>0}$ and some collection $\Bb$ of bounded sets such that $(\R, +, <,\Bb)$ is o-minimal. In particular $\Rr$ is interdefinable with its open core.
\end{enumerate}
\end{proposition}
\begin{proof}
By Theorem~\ref{thm:main-intro}, the open core $\Ro$ is constructible. If $\Rr$ does not define an infinite discrete set then, by Lemma 2.5 in~\cite{DMS10}, $\Ro$ is o-minimal. 
Suppose that $\Rr$ defines an infinite discrete set. By Theorem 2.9 in \cite{DG-burden2}, $\Rr$ is noiseless. Statement (2) then follows from Theorem D in~\cite{Wal22}.
\end{proof}

By \cite[Proposition 3.1]{DG17}, the structure $(\R, +, <, \Z, \Q)$ has dp-rank $3$, and so the dp-rank bound in Proposition~\ref{prop:dprk2} is tight. If $\Rr$ in the proposition is dp-minimal, then it is o-minimal by \cite[Corollary 3.7]{simon-dp-min-order}, and if it defines multiplication, then $\Ro$ is o-minimal by~\cite[Corollary 2.4.]{DG17}. Proposition~\ref{prop:dprk2} fails to generalize to structures of burden $2$. In~\cite[Section 3]{dolich-good22}, an expansion of a definably complete ordered group is constructed which has burden $2$, the independence property, and which defines both an infinite discrete set and a dense and codense set. 

\section{D-minimal structures}\label{sec:d-minimal}

%Usually derived set is defined taking all limit points in the space. Perfect if it is closed and has no isolated points. Cantor–Bendixson theorem states that \underline{closed} sets of a Polish space X can be written uniquely as the disjoint union of a perfect set and a countable set.

A definably complete expansion $\Rr$ of a dense linear order $(R,<)$ is \emph{d-minimal} (short for \emph{discrete-minimal}) if, for every definable family $\Cc$ of subsets of $R$, there exists some $n$ such that every set in $\Cc$ is the union of an open set and at most $n$ discrete sets\footnote{In some texts this definition is called ``strong d-minimality". Meanwhile, ``d-minimality" is reserved for the version of the definition where the $n$ is not necessarily fixed in definable families. In this sense, recent work of Farris and Hieronymi shows that d-minimal expansions of ordered fields are strongly d-minimal.}. This notion originated from the work of Miller~\cite{miller05-tame}. Known examples of $\NIP$ d-minimal structures include $(\R,<,+,\Z)$ and $(\R,+,\cdot,2^\Z)$.  

It is an open question whether every $\NTP$ --or even just $\NIP$-- expansion of $(\R,<,+)$ has d-minimal open core (Question~\ref{Q:d-minimal}). More generally, every constructible expansion of $(\R,<,+)$ known to the author is d-minimal. The role of d-minimality in characterizing neostable structures over the reals is thus conspicuously open. In this section we approach this subject by proving Theorem~\ref{thm:d-minimal-intro} (through Propositions~\ref{prop:shelah-d-min} and~\ref{prop:omega}). 

We fix some useful terminology. Given an expansion $\Rr$ of a linear order $(R,<)$ and a definable set $X\subseteq R$, let $X^{[n]}$ denote the $n$-th Cantor-Bendixson derivative of $X$. That is, $X^{[0]}=X$ and, for every $n$, the set $X^{[n+1]}$ denotes the non-isolated points in $X^{[n]}$. If $\Rr$ is $\omega_1$-saturated, we call $Z=\bigcap_{n<\omega} X^{[n]}$ the \emph{type-definable perfect component of $X$}\footnote{In spite of the chosen terminology, observe that the set $Z$ might not be closed.}. A straightforward application of saturation shows that $Z$ does not contain isolated points.  
%since we are not assuming that X is closed, we are changing slightly the definition of derived set that appears in some sources, which takes all limit points in the space, not just X. 

\begin{remark}\label{rem:discrete-partition}
Given a linear order $(R,<)$, it is well known that a set $X\subseteq R$ is a finite union of at most $n$ discrete sets iff $X^{[n]}=\emptyset$ (see 1.2 and 1.3 in~\cite{FM05-fast-sequences}). In particular, if $X^{[n]}=\emptyset$, then $X$ is partitioned by the discrete sets of the form $X^{[i]} \setminus X^{[i+1]}$, for $i < n$. If $X$ is definable in some expansion of $(R,<)$, then these sets are definable too. Furthermore, given a definable family $\{ X_b : b\in B\}$ of subsets of $R$, for every fixed $i$ the sets $X_b^{[i]} \setminus X_b^{[i+1]}$ are clearly definable uniformly in $b\in B$.  
\end{remark}

We prove two lemmas, the first of which establishes a connection between $\NTP$ and Cantor-Bendixson rank.  

%Discrete case completely ignored in the next lemma. It is trivial and unhelpful.
\begin{lemma}\label{lem:ntp2-CBrank}
Let $\Rr=(R,<,+,\ldots)$ be an $\NTP$ expansion of an ordered group and $X\subseteq R$ be a definable set. Suppose that there exists a definable family $\{ C_b : b\in B\}$ of closed discrete subsets of $R$ satisfying that, for every finite set $F\subseteq X$, there is some $b\in B$ such that $F\subseteq C_b$. Then $X$ is a finite union of discrete sets. 
\end{lemma}
\begin{proof}
Towards a contradiction suppose that $X$ is not a union of finitely many discrete sets, equivalently $X^{[n]}\neq \emptyset$ for every $n$. We may assume that $\Rr$ is $\omega_1$-saturated. Let $Z=\bigcap_{n<\omega} X^{[n]}$ be the type-definable perfect component of $X$. For every finite set $F\subseteq Z$, there exists some $b\in B$ such that $F\subseteq C_b$. Since $Z$ has no isolated points, for every $b\in B$ the set $C_b$ has empty interior in $Z$. By Lemma~\ref{lem:noise-NTP2} we conclude that $\Rr$ has $\TP$.
\end{proof}

%Below discrete case treated below as trivial.
\begin{lemma}\label{lem:discrete-covering}
Let $\Rr=(R,<,+,\ldots)$ be an expansion of an ordered group, and let $X\subseteq R$ be a set that is the union of $n$ discrete sets $X_1,\ldots, X_n$. For every $\varepsilon\in R^{>0}$, let
\[
X(\varepsilon)=\bigcup_{i=1}^n \{ x\in X_i : (x-\varepsilon, x+\varepsilon) \cap X_i = \{x\} \}. 
\]
Then the family $\{ X(\varepsilon) : \varepsilon\in R^{>0}\}$ is an upward directed covering of $X$ by closed discrete sets. 
\end{lemma}
\begin{proof}
The case where $\Rr$ is discrete is trivial so we assume otherwise. 
It is immediate that the described family is an upward directed covering of $X$. Let us fix some $\varepsilon \in R^{>0}$ and let $a\in R$. By density of $(R,<)$ and continuity of the map $x\mapsto 2x$ at $0$, there exists some $\delta\in R^{>0}$ satisfying that $2\delta<\varepsilon$. It is routine to see that, for any $x \leq y$ with $x,y\in (a-\delta, a+\delta)$, it holds that $y\in (x-\varepsilon, x+\varepsilon)$. Consequently, for any $i\leq n$, we have that $|(a-\delta, a+\delta) \cap X(\varepsilon) \cap X_i|\leq 1$, and so $|(a-\delta, a+\delta) \cap X(\varepsilon)|\leq n$. It follows that $X(\varepsilon)$ is closed and discrete. 
\end{proof}

\begin{proposition} \label{prop:shelah-d-min}
Let $\Rr=(R,<,+,\ldots)$ be a d-minimal $\NIP$ expansion of an ordered group. Then the Shelah expansion $\Rshe$ is d-minimal. 
\end{proposition}
\begin{proof}
Let $X\subseteq R$ be a set definable in $\Rshe$. After passing to $X\setminus int(X)$ if necessary we may assume that $X$ has empty interior. We show that $X$ is a finite union of discrete sets. 

By quantifier elimination of the Shelah expansion (see e.g. Proposition 3.24 in~\cite{simon:guide-NIP}) the set $X$ is externally definable. Let $\varphi(x,y)$ be a formula as given by Fact~\ref{honest-definitions}, and set $Y_b=\varphi(R,b)$ for every $b\in R^{|y|}$. Let $B\subseteq R^{|y|}$ be the set of parameters such that $Y_b$ has empty interior. 
Since $\Rr$ is d-minimal, there exists some $n$ such that, for every $b\in B$, the set $Y_b$ is a union of at most $n$ discrete sets. For every $b\in B$ and $i\leq n$, let $Y_{b,i}=Y_b^{[i]} \setminus Y_b^{[i+1]}$. Relying on Remark~\ref{rem:discrete-partition}, for every $b\in B$ and $\varepsilon\in R^{>0}$ let $Y_b(\varepsilon)$ be a closed discrete set as defined in Lemma~\ref{lem:discrete-covering} with respect to $Y_b$ and its partition $\{ Y_{b,0}, \ldots, Y_{b,n}\}$. Observe that the family $\{ Y_{b}(\varepsilon) : b\in B, \, \varepsilon\in R^{>0}\}$ is definable. Furthermore, by Fact~\ref{honest-definitions} and Lemma~\ref{lem:discrete-covering}, for every finite set $F\subseteq X$ there exist parameters $(b,\varepsilon)\in B\times R^{>0}$ such that $F\subseteq Y_{b}(\varepsilon) \subseteq X$. By Lemma~\ref{lem:ntp2-CBrank} we conclude that $X$ is a finite union of discrete sets.
\end{proof}

Any $\NTP$ expansion $\Rr$ of $(\R,<,+)$ by constructible sets is generically locally o-minimal (Theorem~\ref{thm:cons-tame-top}), and so the set of isolated points of any given definable set $X\subseteq \R$ with empty interior is dense in $X$. Consequently, if $\Rr$ were not d-minimal, then it would necessarily have to define a subset of $\R$ with infinite Cantor-Bendixson rank. Perhaps the most obvious candidate for this set would be an order-isomorphic copy of the ordinal $\omega^\omega$. We end the paper by ruling out this candidate. 

\begin{proposition}\label{prop:omega}
Let $\Rr=(R,<,+,\ldots)$ be an expansion of an ordered group. Let $X\subseteq R$ be a definable set such that $(X,<)$ is order-isomorphic to an ordinal $\alpha\geq \omega^\omega$. Then $\Rr$ has $\TP$. 
\end{proposition}
\begin{proof}
For every $\varepsilon\in R^{>0}$, let $X(\varepsilon)=\{ a \in X : (a,a+\varepsilon)\cap X=\emptyset\}$. Since every element in $X$ has a successor in $X$, the definable family $\{ X(\varepsilon) : \varepsilon \in R^{>0}\}$ is an increasing (as $\varepsilon$ decreases) covering of $X$ by discrete sets. Arguing as in the proof of Lemma~\ref{lem:discrete-covering}, it is routine to see that each set $X(\varepsilon)$ is closed. By Lemma~\ref{lem:ntp2-CBrank}, we conclude that $\Rr$ has $\TP$. 
\end{proof}

\begin{question}\label{Q:d-minimal}
Is there an $\NTP$ expansion of $(\R,<,+)$ whose open core is not d-minimal? 
\end{question}

\bibliographystyle{alpha}
\bibliography{tame-open-core}

\end{document}